\documentclass[a4paper,11pt, final]{amsart}
\usepackage{a4wide,mathrsfs,amssymb,amsmath,amsthm,wasysym,dsfont, enumerate}
\usepackage{hyperref}
\usepackage[usenames,dvipsnames]{xcolor}
\hypersetup{colorlinks=true,citecolor=NavyBlue,linkcolor=Brown,urlcolor=Orange}
\usepackage[all,cmtip]{xy}

\newcommand{\C}{\mathbb{C}}
\newcommand{\ZZ}{\mathbb{Z}}

\newcommand{\QQ}{\mathbb{Q}}
\newcommand{\NN}{\mathbb{N}}
\newcommand{\PP}{\mathbb{P}}

\newcommand{\Sy}{\mathfrak S}

\newcommand{\CH}{\operatorname{CH}}
\newcommand{\h}{\mathfrak{h}}

\newcommand{\MM}{\mathcal M}

\DeclareMathOperator{\ide}{id}

\DeclareMathOperator{\ima}{Im}

\newtheorem{theorem}{Theorem}[section]
\newtheorem{claim}[theorem]{Claim}
\newtheorem{lemma}[theorem]{Lemma}

\newtheorem{corollary}[theorem]{Corollary}
\newtheorem{proposition}[theorem]{Proposition}
\newtheorem{conjecture}[theorem]{Conjecture}
\newtheorem{nonumbering}{Theorem}
\newtheorem{nonumberingc}{Corollary}

\newtheorem{nonumberingp}{Proposition}

\theoremstyle{definition}
\newtheorem{remark}[theorem]{Remark}
\newtheorem{definition}[theorem]{Definition}

\begin{document}
	
	\author[Robert Laterveer]
	{Robert Laterveer}

	\address{Institut de Recherche Math\'ematique Avanc\'ee,
		CNRS -- Universit\'e 
		de Strasbourg,\
		7 Rue Ren\'e Des\-car\-tes, 67084 Strasbourg CEDEX,\
		FRANCE}
	\email{robert.laterveer@math.unistra.fr}

	\author[Charles Vial]
	{Charles Vial}
	
	\address{Universit\"at Bielefeld, Fakult\"at f\"ur Mathematik,\
		Postfach 10031,\
		33501 Bielefeld,\
		GERMANY}
	\email{vial@math.uni-bielefeld.de}
	

	\title[On the Chow ring of Cynk--Hulek varieties and Schreieder
	varieties]{On the Chow ring of Cynk--Hulek Calabi--Yau varieties and Schreieder
		varieties}
	
	\begin{abstract} This note is about certain locally complete families of
		Calabi--Yau
		varieties constructed by Cynk and Hulek, and certain varieties
		constructed by Schreieder. We prove that the cycle class
		map on the Chow ring of powers of these varieties admits a section, and that
		these
		varieties admit a multiplicative self-dual Chow--K\"unneth decomposition. As a
		consequence of both results, we prove that the subring of the Chow ring
		generated by divisors, Chern classes, and intersections of two cycles of
		positive codimension injects into cohomology, via the cycle class map.
		We also prove that the small diagonal of Schreieder surfaces admits a
		decomposition similar to that of K3 surfaces.
		As a by-product of our main result, we verify a conjecture of Voisin
		concerning zero-cycles on
		the
		self-product of Cynk--Hulek Calabi--Yau varieties, and in the odd-dimensional
		case we verify a
		conjecture of Voevodsky concerning smash-equivalence. Finally, in positive
		characteristic, we show that the supersingular Cynk--Hulek
		Calabi--Yau
		varieties provide examples of Calabi--Yau varieties with ``degenerate''
		motive.
	\end{abstract}

	
	\maketitle

	\section*{Introduction}

	In the course of a quest for Calabi--Yau varieties that are modular, Cynk and
	Hulek \cite{CH} constructed certain Calabi--Yau varieties $X$ of arbitrary
	dimension $n$ over $\C$. Their construction starts from a product of $n$
	complex
	elliptic curves $E_1,\ldots, E_n$. The Calabi--Yau variety $X$ is obtained by
	considering
	\[\xymatrix{
		& E_1\times\cdots\times E_n \ar[d]^p \\
		X\ \ \ \ar[r]^{f\hspace{2cm} } &  \ \ \ (E_1\times\cdots\times E_n)/G :=
		\bar X   
	}\]
	where $G$ is a certain group of automorphisms (specifically $G\cong
	\ZZ_2^{n-1}$, or $G\cong \ZZ_3^{n-1}$ and $E_1=\cdots=E_n$ is an elliptic curve
	with an order-3 automorphism), and $f$ is a crepant resolution of
	singularities. We refer to Theorems \ref{ch} and \ref{ch3}
	below for explicit definitions, and to Propositions \ref{p:inductionZ2} and
	\ref{p:inductionZ3} together with the proof of Claim \ref{key} for an explicit
	construction.\medskip
	
	The difference between the two types of Cynk--Hulek varieties ($G=
	\ZZ_2^{n-1}$, resp. $G= \ZZ_3^{n-1}$) is illustrated by their Hodge diamond.
	In the first case (\emph{i.e.} $G=
	\ZZ_2^{n-1}$), the Hodge diamond looks like
	\[ \begin{array}[c]{ccccccc}
	&&&1&&&\\
	&&  &\ast&  &&\\
	&&&\vdots&&&\\
	1&\ast&\dots&\dots&\dots&\ast&1\\
	&&&\vdots&&&\\
	&&  &\ast&  &&\\
	&&&1&&&\\
	\end{array}\]	
	(where $\ast$ means some unspecified number, and all empty entries are $0$),
	whereas for the second case (\emph{i.e.} $G= \ZZ_3^{n-1}$), the Hodge diamond
	is
	\[ \begin{array}[c]{ccccccccc}
	&&&&1&&&&\\
	&&&  &\ast&  &&&\\
	&&&&\vdots&&&&\\
	1&0&\dots&0&\ast&0&\dots&0&1\\
	&&&&\vdots&&&&\\
	&&&  &\ast&  &&&\\
	&&&&1&&&&\\
	\end{array}\]

	Recently, Stefan Schreieder \cite{Schreieder} generalized the construction of
	Cynk--Hulek in order  to solve some construction problems for Hodge numbers.
	His
	construction starts with the hyperelliptic curve $C$ which is the smooth
	projectivization of the affine curve $\{y^2 = x^{3^c} - 1\}$ equipped with the
	action of a primitive $3^c$-th root of unity $\zeta$ acting as $(x,y) \mapsto
	(\zeta\cdot x,y)$. The variety $X$ is an explicit smooth projective birational
	model of $C^n/G$, where $G$ is a certain subgroup of $(\ZZ_{3^c})^n$
	isomorphic to  $(\ZZ_{3^c})^{n-1}$ (see Proposition \ref{prop:schreieder} and
	the proof of Claim \ref{key}). 
	The Hodge diamond of a Schreieder variety looks like
	\[ \begin{array}[c]{ccccccccccccccc}
	&&&&&&&1&&&&&&&\\
	&&&&&&  &\ast&  &&&&&&\\
	&&&&&&&\vdots&&&&&&&\\
	&&&&&&&\vdots&&&&&&&\\
	0&\dots&	0&g&0&\dots&0&\ast&0&\dots&0&g&0&\dots&0\\
	&&&&&&&\vdots&&&&&&&\\
	&&&&&&&\vdots&&&&&&&\\
	&&&&&&  &\ast&  &&&&&&\\
	&&&&&&&1&&&&&&&\\
	\end{array}\]	
	where $g=(3^c-1)/2$ can occur at any desired place $h^{a,b}$ with $a+b=n$.	
	
	From an arithmetic perspective, the
	construction of Schreieder has been used by Flapan and Lang \cite{FL} to 
	construct motives associated to certain algebraic Hecke characters, thereby
	generalizing the modularity result of Cynk and Hulek \cite{CH}. \medskip
	
	The varieties of Cynk--Hulek and of Schreieder are thus both very special from
	a Hodge-theoretic point of view and from an arithmetic point of view.
	The aim of this note is to confirm that these varieties are also very
	special from a cycle-theoretic point of view.\medskip
	
	Let $\CH^i(X)$ denote the Chow groups with
	rational coefficients, let $\CH^i_{num}(X)$ denote the subgroup of
	numerically
	trivial cycles, and let $\overline{\CH}^i(X)$ denote the quotient. 
	Our main result concerns the multiplicative structure of the Chow ring
	of~$X$\,:

	\begin{nonumbering}[Theorem \ref{t:bv}] Let $X$ be either a Cynk--Hulek
		Calabi--Yau variety as in Theorem~\ref{ch} or \ref{ch3}, or a Schreieder
		variety as in Theorem \ref{schreieder}.  Then the
		$\QQ$-algebra epimorphism $\CH^*(X^m) \to \overline{\CH}^*(X^m)$ admits a
		section whose image contains the Chern classes of $X^m$, for all positive
		integers $m$.
		
		Moreover, assuming $n:=\dim X \geq 2$,  the graded subalgebra $R^*(X)
		\subseteq
		\CH^*(X)$ generated by
		divisors, Chern classes and by cycles that are the intersection of two cycles
		in
		$X$ of positive codimension injects into $\overline{\CH}^*(X)$. In particular,
		for any $k$
		the image of the  intersection map 
		$$\CH^i(X) \otimes \CH^{k-i}(X) \to
		\CH^{k}(X)\ \ \ \ (0<i<k)\ $$ 
		injects into cohomology.
	\end{nonumbering}         
	
	Theorem \ref{t:bv} is similar to results in the Chow ring of K3 surfaces
	\cite{BV}, and is closely related to the conjectural ``splitting property'' of
	Beauville \cite{Beau3}. Presumably, the fact that $R^{n}(X) = \QQ
	c_{n}(X)$ is true for {\em any\/} Calabi--Yau variety\footnote{This expectation
		is perhaps overly optimistic. 
		Voisin \cite[p. 101]{Vo} writes more prudently\,: ``It would be very
		interesting to understand the class of Calabi--Yau varieties satisfying
		conclusions analogous to'' Calabi--Yau complete intersections. Bazhov
		\cite{Baz} states (and
		proves in certain cases) a weaker version of this expectation, only
		considering $0$-cycles that are intersections of divisors.}\,; for instance,
	this
	was
	established for Calabi--Yau complete intersections \cite{V13}, \cite{LFu}. On
	the other hand, the full statement of Theorem \ref{t:bv} is certainly {\em
		not\/} true for all Calabi--Yau varieties \cite[Example 2.1.5]{Beau3}\,; this
	behavior is peculiar to the Cynk--Hulek Calabi--Yau varieties.
	
	Somewhat surprisingly, the Schreieder varieties give examples in any dimension,
	and with arbitrarily large geometric genus, for which the
	intersection product in the Chow ring is ``as degenerate as possible''. (This
	should be contrasted with the behavior of the surfaces $S\subset\PP^3$
	exhibited in \cite{OG}, for which the rank of $\ima\bigl( \CH^1(S)\otimes
	\CH^1(S)\to \CH^2(S)\bigr)$ gets arbitrarily large when the degree of $S$
	grows.)
	Schreieder surfaces of genus $1$ are K3 surfaces while Schreieder surfaces of
	higher genus are modular elliptic of Kodaira dimension 1 (see \cite{Fla} and
	Remark~\ref{rmk:schsur}). For those, we obtain as a corollary the existence of a
	decomposition of the small diagonal similar to that of K3 surfaces proved by
	Beauville and Voisin \cite{BV}\,:
	
	\begin{nonumberingc}\label{cor:moddiag}
		Let $S$ be a Schreieder surface. Then there exists a point $p\in S$ such that
		$$(x,x,x) - (x,x,p) - (x,p,x) - (p,x,x) +(p,p,x) + (p,x,p) + (x,x,p) = 0 \quad
		\mbox{in}\ \CH^4(S\times S \times S).$$
		Here $(x,x,x), (x,x,p), (x,p,p)$  are the classes of the images of $S$ into
		$S\times S \times S$ by the maps $x \mapsto (x,x,x)$, $x \mapsto (x,x,p)$, $x
		\mapsto (x,p,p)$.
	\end{nonumberingc}

	In order to show that the $\QQ$-algebra epimorphism $\CH^*(X^m) \to
	\overline{\CH}^*(X^m)$ of Theorem \ref{t:bv} admits a section, we prove that
	$X$ satisfies a certain
	condition $(\star)$ (\emph{cf.} Definition \ref{def:Star}) which was introduced
	in \cite{FV}.
	The ``moreover'' part of Theorem \ref{t:bv} is not a formal consequence of
	the existence of a section, and is obtained, via Proposition
	\ref{prop:crucial}, by computing the motive of $X$
	and by establishing yet another result related to the splitting property (see
	\S
	\ref{sec:mck} for the notion of \emph{multiplicative Chow--K\"unneth
		decomposition})\,:
	
	\begin{nonumbering}[Theorem \ref{t:mck}]Let $X$ be either a Cynk--Hulek
		Calabi--Yau variety as in Theorem~\ref{ch} or \ref{ch3}, or a
		Schreieder variety as in Theorem \ref{schreieder}. 
		Then $X$ admits a multiplicative self-dual Chow--K\"unneth decomposition, in
		the sense of \cite{SV}.
	\end{nonumbering}
	
	Other varieties admitting a multiplicative Chow--K\"unneth decomposition
	include abelian varieties, hyperelliptic curves, Hilbert schemes of points of
	K3 surfaces and of abelian surfaces \cite{V6}, and generalized Kummer
	varieties \cite{FTV}.
	Theorem \ref{t:mck} provides the first examples of Calabi--Yau varieties of
	dimension $>2$ with a multiplicative Chow--K\"unneth decomposition, while
	Theorem \ref{t:bv} provides the first examples of Calabi--Yau varieties of
	dimension $>2$  for which the subalgebra of the Chow ring generated by divisors
	and the Chern classes
	injects into cohomology via the cycle class map.
	\medskip

	Along the way,
	we compute (Corollary \ref{C:CH}) the Chow motive of
	certain finite quotients of products of (hyper)elliptic curves (including the
	quotients considered in Theorems \ref{ch}, \ref{ch3} and \ref{schreieder}), but
	also compute (Claim~\ref{key2}) the Chow motive of the Cynk--Hulek Calabi--Yau
	varieties and of the Schreieder varieties.
	In Section~\ref{sec:applications}, we offer three applications. 
	
	First, we use Corollary~\ref{C:CH} to establish\,:

	\begin{nonumbering}[Theorem \ref{conjV}] Let $X$ be a Cynk--Hulek Calabi--Yau
		variety of dimension $n$ as in Theorem \ref{ch} or \ref{ch3}. Then any
		$a,a^\prime\in \CH^n_{num}(X)$ satisfy
		\[ a\times a^\prime = (-1)^n \, a^\prime\times a\ \ \ \hbox{in}\
		\CH^{2n}(X\times X)\ .\]
	\end{nonumbering} 
	According to an old conjecture of Voisin (\cite{V9}, \emph{cf.} also Section
	\ref{svois} below), the statement of  Theorem~\ref{conjV} should hold for {\em
		any\/} Calabi--Yau variety. As far as we are aware, Theorem~\ref{conjV} gives
	the first examples of Calabi--Yau varieties of arbitrary dimension verifying
	Voisin's conjecture.\medskip
	
	Second, a consequence of Claim~\ref{key2} concerns Voevodsky's conjecture on
	smash-equivalence\,;
	we refer to~\cite{Voe} and Section \ref{svoe} below for the definition and
	background of smash-equivalence.
	
	\begin{nonumberingp}[Proposition \ref{conjVoe}] Let $X$ be either a Cynk--Hulek
		Calabi--Yau variety as in Theorem \ref{ch} or \ref{ch3}, or a Schreieder
		variety as in Theorem \ref{schreieder}. Assume that $X$ is odd-dimensional.
		Then smash-equivalence and
		numerical equivalence coincide
		for all $\CH^i(X)$.
	\end{nonumberingp}

	Finally, in a brief excursion to positive characteristic, we exhibit, as a
	consequence of Claim~\ref{key2}, examples of Calabi--Yau varieties in
	characteristic $\geq 5$ whose motive is ``degenerate''\,:
	
	\begin{nonumberingp}[Proposition \ref{ss}] Let $k$ be an algebraically closed
		field of characteristic
		$\ge 5$. Let $X$ be a Cynk--Hulek Calabi--Yau variety over $k$ as in Theorem
		\ref{ch} or \ref{ch3}, where the elliptic curves are assumed to be
		supersingular and $X$ is even-dimensional. Then the Chow motive of $X$ is
		isomorphic to a direct sum of Lefschetz motives. Consequently,  the cycle class
		map to $\ell$-adic cohomology induces
		isomorphisms
		\[  \CH^i(X)_{\QQ_\ell}\ \xrightarrow{\cong}\ H^{2i}(X, \QQ_\ell(i))\ \ \
		\forall i\ \]
		(where $\ell$ is a prime different from $\hbox{char}\, k$).
	\end{nonumberingp}
	
	\noindent \textbf{Acknowledgments.} Thanks to Julius Ross, for asking whether
	there exist Calabi--Yau varieties of arbitrarily high dimension for which
	Voisin's conjecture is known. His question sparked this project.

	\section{The varieties of Cynk--Hulek and Schreieder}
	
	We denote $\ZZ_n$ the cyclic group of order $n$. Given a smooth projective
	variety, we denote $H_{tr}(X)$ its transcendental cohomology\,; it is the
	orthogonal complement (with respect to the choice of a polarization) of the
	subspace spanned by algebraic classes.

	\subsection{The Cynk--Hulek construction}

	\begin{theorem}[Cynk--Hulek \cite{CH}]\label{ch} 
		Let $E_1,\ldots,E_n$ be
		elliptic curves. For any $n\in\NN$, let 
		\[ G = \{  (m_1,\ldots,m_n)\ \in \ZZ_2^n :  m_1+\cdots+m_n=0 \}\cong
		\ZZ_2^{n-1} \]
		act on $E_1\times\cdots\times E_n$, where the generator of $\ZZ_2$  acts on
		$E_i$ by the $[-1]$-involution. Then there exists a crepant resolution
		\[ f\colon\  X\ \to\ \bar{X}:=(E_1\times\cdots\times E_n)/G\ ,\]
		and so $X$ is a Calabi--Yau variety. Moreover, such Calabi--Yau varieties
		form a locally complete family.
	\end{theorem}
	
	\begin{proof} This is \cite[Corollary 2.3]{CH}. In fact, the crepant resolution
		$X$ can be constructed explicitly, inductively on the number of elliptic
		curves,
		\emph{cf.} the proof of Proposition \ref{p:inductionZ2} below. That  such
		Calabi--Yau varieties form a
		locally complete family can be seen as follows\,: since elliptic curves have a
		one-dimensional deformation family, $X$ clearly fits into an $n$-dimensional
		deformation family. On the other hand, $H^n(X)$ is isomorphic to
		$H^1(E_1)\otimes
		\cdots\otimes H^1(E_n)$ plus possibly some algebraic classes, and in
		particular $h^{1,n-1}(X)=n$\,; see
		\cite[Lemma~2.4]{CH} or Corollary~\ref{C:CH} below. By Serre duality
		$H^1(X,\mathcal{T}_X) \cong H^{n-1}(X,\Omega_X^1)$ so that $\dim
		H^1(X,\mathcal{T}_X)  = n$.
	\end{proof}
	
	In the case of elliptic curves with extra endomorphisms (precisely,
	automorphisms of order~3), Cynk and Hulek construct examples of Calabi--Yau
	varieties with cohomology ``as simple as possible''.
	
	\begin{theorem}[Cynk--Hulek \cite{CH}]\label{ch3} Let $E$ be an elliptic curve
		with an order~$3$ automorphism~$\nu$. 
		For any $n\in\NN$, let 
		\[ G = \{  (m_1,\ldots,m_n)\ \in \ZZ_3^n : m_1+\cdots+m_n=0 \}\cong
		\ZZ_3^{n-1} \]
		act on $E^n$ by $\nu^{m_i}$ on the $i$-th factor. There exists a crepant
		resolution
		\[ f\colon\  X\ \to\ \bar{X}:=E^n/G\ ,\]
		and so $X$ is a Calabi--Yau variety. Moreover, for $n>2$, such Calabi--Yau
		varieties are rigid, and their transcendental cohomology has Hodge numbers
		$h_{tr}^{p,q} = 1$ if $\{p,q\} = \{n,0\}$, and $h_{tr}^{p,q} = 0$ otherwise.
	\end{theorem}

	\begin{proof} This is \cite[Theorem 3.3]{CH} (the construction of $X$ is also
		explained in \cite[Section 5.3]{HKS}).  In fact, the crepant resolution
		$X$ can be constructed explicitly, inductively on the number of elliptic
		curves,
		\emph{cf.} the proof of Proposition \ref{p:inductionZ3} below.	
		
		Arguing as in the proof of Theorem
		\ref{ch}, we see that such $X$ is rigid because $h^{1,n-1}(X) = 0$\,; see
		\cite[Theorem~3.3]{CH} or Corollary~\ref{C:CH} below.
	\end{proof}
	
	\begin{remark} The Cynk--Hulek varieties $X$ of Theorem \ref{ch3} are
		$N^1$--maximal, in the sense of \cite{BLP}\,; this means that $\dim
		H^n_{tr}(X,\QQ)=2$.
	\end{remark}

	\subsection{The Schreieder construction}
	
	By using iterated resolutions of $\ZZ_3$-quotient singularities, Schreieder
	generalizes (see however Remark \ref{diff0}) the Cynk--Hulek construction of
	Theorem \ref{ch3} and proves the
	following theorem.
	\begin{theorem}[Schreieder \cite{Schreieder}]\label{schreieder}
		Let $c$ be a positive integer, and let $\zeta$ be a primitive $3^c$-th roof of
		unity. Let $C$ be the smooth projective hyperelliptic curve obtained as the
		smooth projectivization of the affine curve $\{y^2 = x^{3^c}+1\}$. Endow $C$
		with the action of $\ZZ_{3^c}$ given by $(x,y)\mapsto (\zeta\cdot x,y)$. For
		any
		$n\in \NN$ and any integers $a,b \geq 0$ such that $a>b$
		and $a+b = n$, let 	
		\[ G_{a,b} = \{  (m_1,\ldots,m_n)\ \in \ZZ_{3^c}^n : m_1+\cdots+ m_a - m_{a+1}
		-
		m_{a+b}=0 \}\cong
		\ZZ_3^{n-1} \]
		act on $C^n$ by $\zeta^{m_i}$ on the $i$-th factor.
		Then $C^n / G_{a,b}$ admits a smooth projective model $X$ whose 
		transcendental
		cohomology has Hodge numbers $h_{tr}^{p,q} = (3^c-1)/2$ if $\{p,q\} =
		\{a,b\}$,
		and $h_{tr}^{p,q} = 0$ otherwise.
	\end{theorem}
	
	Schreieder provides in \cite[\S 8]{Schreieder} an explicit construction of $X$.
	The construction is inductive on the number of factors $C$, and is recalled in
	\S \ref{sec:3c}. When referring to the ``Schreieder varieties'', we will mean
	those explicit models. 
	
	\begin{remark} \label{rmk:schsur} A Schreieder variety of dimension $2$ is a K3
		surface when $c=1$
		(these K3 surfaces have been intensively studied by Shioda--Inose \cite{SI}),
		and is an elliptic modular surface of Kodaira dimension $1$ for all $c>1$
		\cite[Theorems 3.2 \& 9.2]{Fla}. These surfaces are very special\,: they are
		$\rho$--maximal (in the sense of \cite{Beau}) and have Mordell--Weil rank $0$
		\cite[Corollary 6.1]{Fla}.
		
	\end{remark}

	\begin{remark}\label{diff0} In case $c=1$ and $b=0$, the Schreieder variety
		$X_S$ (given by Theorem \ref{schreieder}) and the Cynk--Hulek variety $X_{CH}$
		(given by Theorem \ref{ch3}) are both resolutions of the same singular variety
		$C^n/G_{n,0}$. They share the same Hodge numbers $h^{p,q}$ for $p\not=q$, but
		they are (\emph{a priori}) different\,; indeed, $X_{CH}$ is Calabi--Yau,
		whereas
		$X_S$ is only ``numerically Calabi--Yau''. The difference in the construction
		of
		$X_S$ and $X_{CH}$ is outlined in Remark \ref{diff}.
	\end{remark}

	\section{Multiplicative Chow--K\"unneth decompositions and distinguished
		cycles}
	
	The aim of this section is to recall briefly the notions of
	\emph{multiplicative Chow--K\"unneth decomposition}, and of \emph{distinguished
		cycle} on varieties with motive of abelian type. Combining both notions, we
	reduce the proof of the main Theorem \ref{t:bv} to showing that the
	transcendental cohomology $H^i_{tr}(X)$ is concentrated in degree $i= \dim X$,
	and that
	the motive of $X$ satisfies a certain condition $(\star)$ (Definition
	\ref{def:Star})\,; \emph{cf.} Proposition \ref{prop:crucial} and the final
	Remark \ref{rmk:reduction}.
	
	\subsection{Multiplicative Chow--K\"unneth decompositions}\label{sec:mck}
	
	\begin{definition}[Murre \cite{Mur}]\label{ck} Let $X$ be a smooth projective
		variety of dimension $n$. We say that $X$ has a 
		{\em Chow--K\"unneth  decomposition\/} if there exists a decomposition of the
		diagonal
		\[ \Delta_X= \pi^0_X+ \pi^1_X+\cdots +\pi^{2n}_X\ \ \ \hbox{in}\
		\CH^n(X\times X)\ ,\]
		such that the $\pi^i_X$ are mutually orthogonal idempotents and
		$(\pi^i_X)_\ast H^\ast(X)= H^i(X)$.
		Given a Chow--K\"unneth decomposition for $X$, we set 
		$$\CH^i(X)_{(j)} := (\pi_X^{2i-j})_*\CH^i(X).$$
		The Chow--K\"unneth decomposition is said to be {\em self-dual\/} if
		\[ \pi^i_X = {}^t \pi^{2n-i}_X\ \ \ \hbox{in}\ \CH^n(X\times X)\ \ \ \forall
		i\ .\]
		(Here ${}^t \pi$ denotes the transpose of a cycle $\pi$.)
	\end{definition}
	
	\begin{remark} \label{R:Murre} The existence of a Chow--K\"unneth decomposition
		for any smooth projective variety is part of Murre's conjectures \cite{Mur}.
		It is expected that for any $X$ with a Chow--K\"unneth
		decomposition, one has
		\begin{equation*}\label{hope} \CH^i(X)_{(j)}\stackrel{??}{=}0\ \ \ \hbox{for}\
		j<0\ ,\ \ \ \CH^i(X)_{(0)}\cap \CH^i_{num}(X)\stackrel{??}{=}0.
		\end{equation*}
		These are Murre's conjectures B and D, respectively.
	\end{remark}

	\begin{definition}[Definition 8.1 in \cite{SVfourier}]\label{mck} Let $X$ be a
		smooth
		projective variety of dimension $n$. Let $\delta_X^{}\in \CH^{2n}(X\times
		X\times X)$ be the class of the small diagonal
		\[ \delta_X^{}:=\bigl\{ (x,x,x) : x\in X\bigr\}\ \subset\ X\times
		X\times X\ .\]
		A Chow--K\"unneth decomposition $\{\pi^i_X\}$ of $X$ is {\em multiplicative\/}
		if it satisfies
		\[ \pi^k_X\circ \delta_X^{}\circ (\pi^i_X\otimes \pi^j_X)=0\ \ \ \hbox{in}\
		\CH^{2n}(X\times X\times X)\ \ \ \hbox{for\ all\ }i+j\not=k\ .\]
		In that case,
		\[ \CH^i(X)_{(j)}:= (\pi_X^{2i-j})_\ast \CH^i(X)\]
		defines a bigraded ring structure on the Chow ring\,; that is, the
		intersection product has the property that 
		\[  \ima \Bigl(\CH^i(X)_{(j)}\otimes \CH^{i^\prime}(X)_{(j^\prime)}
		\xrightarrow{\cdot} \CH^{i+i^\prime}(X)\Bigr)\ \subseteq\ 
		\CH^{i+i^\prime}(X)_{(j+j^\prime)}\ .\]
	\end{definition}

	The property of having a multiplicative Chow--K\"unneth decomposition is
	severely restrictive, and is closely related to Beauville's ``(weak) splitting
	property'' \cite{Beau3}. For more ample discussion, and examples of varieties
	admitting a multiplicative Chow--K\"unneth decomposition, we refer to
	\cite[Chapter 8]{SVfourier}, as well as \cite{V6}, \cite{SV},
	\cite{FTV}.\medskip

	\subsection{Distinguished cycles on varieties with motive of abelian type}
	
	The following crucial notion was introduced by O'Sullivan \cite{o's}.
	
	\begin{definition}[Symmetrically distinguished cycles on abelian varieties
		\cite{o's}]\label{def:SD}
		Let $A$ be an abelian variety and $\alpha\in \CH^*(A)$. For each integer
		$m\geq
		0$, denote by $V_{m}(\alpha)$ the $\QQ$-vector subspace of $\CH^*(A^{m})$
		generated
		by elements of the form
		$$p_{*}(\alpha^{r_{1}}\times \alpha^{r_{2}}\times \cdots\times
		\alpha^{r_{n}}),$$
		where $n\leq m$, $r_{j}\geq 0 $ are integers, and $p : A^{n}\to A^{m}$ is a
		closed immersion with each component $A^{n}\to A$ being either a projection or
		the composite of a projection with $[-1]: A\to A$. Then $\alpha$ is
		\emph{symmetrically distinguished} if for every $m$ the restriction of the
		projection $\CH^*(A^{m})\to \overline\CH^*(A^{m})$ to $V_{m}(\alpha)$ is
		injective.
	\end{definition}
	
	The main result of \cite{o's} is\,:
	
	\begin{theorem}[O'Sullivan \cite{o's}]\label{thm:SD}
		Let $A$ be an abelian variety. Then $\operatorname{DCH}^*(A)$, the
		symmetrically distinguished cycles in
		$\CH^*(A)$, form a graded sub-$\QQ$-algebra that contains symmetric divisors
		and that is
		stable under pull-backs and push-forwards along homomorphisms of abelian
		varieties. Moreover
		the composition $$\operatorname{DCH}^*(A)\hookrightarrow
		\CH^*(A)\twoheadrightarrow \overline
		\CH^*(A)$$ is an isomorphism of $\QQ$-algebras.
	\end{theorem}
	
	Let $X$ be a smooth projective variety such that its Chow motive $\h(X)$
	belongs
	to the strictly full and thick subcategory of Chow motives generated by the
	motives of abelian varieties. We say that $X$ has motive \emph{of abelian
		type}.
	A \emph{marking} for $X$ is an isomorphism $\phi:
	\h(X)\stackrel{\simeq}{\longrightarrow} M$ of Chow motives with $M$ a direct
	summand of a Chow motive of the form 
	$\oplus_{i} \h(A_{i})(n_{i})$ cut out by an idempotent matrix $P \in
	\mathrm{End}(\oplus_i \h(A_i)(n_i))$ whose entries are symmetrically
	distinguished cycles, where $A_{i}$ is an abelian variety and $n_{i}$ is an
	integer (the Tate twist). We refer to \cite[Definition~3.1]{FV} for a precise
	definition.
	Given a marking $\phi: \h(X)\stackrel{\simeq}{\longrightarrow} M$, we define
	the subgroup of \emph{distinguished cycles} of $X$, denoted 
	$\operatorname{DCH}^*_{\phi}(X)$ to be the pre-image of
	$\operatorname{DCH}^*(M):= P_*\bigoplus_i \operatorname{DCH}^{*-n_i}(A_i)$
	via the induced isomorphism
	$\phi_{*}:\CH^*(X)\stackrel{\simeq}{\longrightarrow}\CH^*(M)$. 
	Given another smooth projective variety $Y$ with a marking $\psi : \h(Y) \to
	N$, the tensor product $\phi \otimes \psi : \h(X\times Y) \to M\otimes N$
	defines naturally a marking for $X\times Y$. A morphism $f: X\to Y$ will be said
	to be a \emph{distinguished morphism} if its graph is distinguished with respect
	to the product marking $\phi\otimes \psi$.

	The composition $$\operatorname{DCH}^*_{\phi}(X)\hookrightarrow
	\CH^*(X)\twoheadrightarrow \overline\CH^*(X)$$ is clearly bijective. In
	other words, $\phi$ provides a section (as graded vector spaces) of the natural
	projection $\CH^*(X)\twoheadrightarrow\overline\CH^*(X)$. In \cite{FV}, we
	found
	sufficient conditions on the marking $\phi$ for
	$\operatorname{DCH}^*_{\phi}(X)$
	to define a $\QQ$-subalgebra of $\CH^*(X)$\,:
	
	\begin{definition}[Definition 3.7 in \cite{FV}]\label{def:Star} 
		We say that the marking $\phi: \h(X)\stackrel{\simeq}{\longrightarrow} M$
		satisfies the condition $(\star)$ if the following two conditions are
		satisfied\,: 
		\begin{itemize}
			\item[$(\star_{\mathrm{Mult}})$] the small diagonal $\delta_{X}$ belongs to
			$\operatorname{DCH}^*_{\phi^{\otimes 3}}(X^{3})$\,; that is, under the
			induced
			isomorphism $\phi^{\otimes 3}_{*}:
			\CH^*(X^{3})\stackrel{\simeq}{\longrightarrow} \CH^*(M^{\otimes 3})$, the
			image
			of $\delta_{X}$ is symmetrically distinguished, \emph{i.e.} in
			$\operatorname{DCH}^*(M^{\otimes 3})$.
			\item[$(\star_{\mathrm{Chern}})$]  all Chern classes $c_{i}(X)$ belong to
			$\operatorname{DCH}^*_{\phi}(X)$\,;
		\end{itemize}
		If in addition $X$ is equipped with the action of a finite group $G$, we say
		that the  marking $\phi: \h(X)\stackrel{\simeq}{\longrightarrow} M$  satisfies
		$(\star_G)$ if\,:
		\begin{itemize}
			\item[$(\star_G)$] the graph $g_X$ of $g: X\to X$ belongs to
			$\operatorname{DCH}^*_{\phi^{\otimes 2}}(X^{2})$ for all $g\in G$. 
		\end{itemize}
	\end{definition}
	
	\begin{proposition}[Proposition 3.12 in \cite{FV}]\label{prop:Disting}
		If the marking $\phi: \h(X)\stackrel{\simeq}{\longrightarrow} M$ satisfies the
		condition $(\star)$, then there is a section, as \emph{graded algebras}, for
		the
		natural surjective morphism $\CH^*(X)\to \overline\CH^*(X)$ such that all
		Chern
		classes of $X$ are in the image of this section.
		
		In other words, under $(\star)$, we have a graded $\QQ$-sub-algebra
		$\operatorname{DCH}_\phi^*(X)$ of the Chow ring $\CH^*(X)$, which contains all
		the Chern
		classes of $X$ and is mapped isomorphically to $\overline\CH^*(X)$. Elements
		of
		$\operatorname{DCH}_\phi^*(X)$ are called \emph{distinguished cycles}.
	\end{proposition}
	
	We refer to \cite{FV} for 
	example of varieties satisfying $(\star)$\,; for our purpose here, we mention
	that these include abelian varieties, hyperelliptic curves (see Proposition
	\ref{prop:markinghyperelliptic}), and varieties with trivial Chow
	groups\footnote{A smooth projective variety $X$ is said to have \emph{trivial
			Chow groups} if the
		Chow groups of $X$ base-changed to a universal domain are finite-dimensional
		$\QQ$-vector spaces.}.
	The property $(\star)$ is very flexible\,; in \cite[Section 4]{FV}, it is shown
	that this property is stable under product, projectivization of vector bundles,
	and blow-ups, under certain conditions on some Chern classes.  Those will be
	utilized in the proof of our main theorems where the smooth models will be
	obtained by blowing up subvarieties with trivial Chow groups inside a
	product of hyperelliptic curves, taking finite quotients and iterating\,; see
	the arguments in Section \ref{s:main}. For the record, let us write down
	explicitly one of the results of \cite{FV}\,:
	\begin{proposition}[Propositions 4.5 and 4.8 in \cite{FV}]
		Let $X$ be a smooth projective variety and let $i:Y\hookrightarrow X$ be a
		closed smooth
		subvariety. Let $\tilde X$ be the blow-up of $X$ along $Y$ and let $E$ be the
		exceptional divisor, so that we have a commutative diagram
		$$	\xymatrix{
			E\ar@{^{(}->}[r]^{j} \ar[d]_{p} & \tilde X\ar[d]^{\tau}\\
			Y \ar@{^{(}->}[r]^{i} & X
		}$$
		If we have markings satisfying the condition $(\star)$ for $X$ and
		$Y$ such that $i:Y\hookrightarrow X$ is distinguished, then $E$ and $\tilde X$
		have  natural markings that satisfy $(\star)$ and are
		such that the morphisms $i,j,\tau$ and $p$ are all
		distinguished. 
		
		If in addition $X$ is equipped with the action of a finite group $G$ such that
		$G\cdot Y = Y$ and such that the markings of $X$ and $Y$ satisfy $(\star_G)$,
		then the natural markings of $E$ and $\tilde X$ also satisfy $(\star_G)$.\qed
	\end{proposition}

	Theorem \ref{t:bv} and Theorem \ref{t:mck}  are related by\,:

	\begin{proposition}[Proposition 6.1 in \cite{FV}]\label{prop:multmarking}
		Let $X$ be a smooth projective variety with a marking $\phi$ that satisfies
		$(\star_{\mathrm{Mult}})$. Then $X$ has a self-dual multiplicative
		Chow--K\"unneth decomposition, consisting of distinguished cycles in
		$\CH^*(X\times X)$, with the property that
		$$\operatorname{DCH}_{\phi}^*(X) \subseteq \CH^*(X)_{(0)}$$ 
		(and equality holds
		for $*=0,1,\dim X-1, \dim X$).\qed
	\end{proposition}
	
	Thus if we have  a smooth projective variety $X$ with a marking $\phi$ that
	satisfies
	$(\star_{\mathrm{Mult}})$, we have a chain of homomorphisms
	$$\operatorname{DCH}_{\phi}^*(X) \hookrightarrow \CH^*(X)_{(0)} \hookrightarrow
	\CH^*(X) \twoheadrightarrow \overline{\CH}^*(X),$$
	whose composition is an isomorphism, and where the left inclusion arrow is
	conjecturally an isomorphism (by Murre's conjecture D).

	\subsection{A crucial proposition}
	
	The following proposition is crucial to the proof of the second part of Theorem
	\ref{t:bv}.
	
	\begin{proposition}\label{prop:crucial}
		Let $X$ be a smooth projective variety of dimension $n\geq 2$ that admits a
		marking
		satisfying the condition $(\star)$ of Definition \ref{def:Star}. Assume
		that
		the cohomology of $X$ is spanned by algebraic classes in degree $\neq n$. Then
		the graded subalgebra $R^*(X) \subseteq \CH^*(X)$ generated by divisors, Chern
		classes and by cycles that are the intersection of two cycles in $X$ of
		positive
		codimension  injects into $\overline{\CH}^*(X)$.
	\end{proposition}      
	\begin{proof} Fix a marking $\phi : \h(X) \to M$ that satisfies $(\star)$\,; in
		particular, $X$ has motive of abelian type. 
		First we exploit the condition on the cohomology of $X$. This condition means
		that $H^{2i+1}(X) = 0$ for $2i+1\neq n$, and that the cycle class map
		$\CH^i(X)
		\to \mathrm{H}^{2i}(X)$ is surjective for $2i \neq n$. By using the
		nondegenerate cup-product pairing between $\mathrm{H}^{2i}(X)$ and
		$\mathrm{H}^{2n-2i}(X)$, we obtain a K\"unneth decomposition of the diagonal
		$\Delta_X = p^n_X + \sum_{2i\neq n} p^{2i}_X \in \mathrm{H}^{2n}(X\times X)$,
		where the classes $p^j_X$ are algebraic and such that the homological motives
		$(X,p^{2i}_X)$ are isomorphic to a direct sum of copies of the Lefschetz
		motive
		$\mathds{1}(-i)$ for $2i\neq n$, and where $(p_X^n)_*\mathrm{H}^*(X) = 
		\mathrm{H}^n(X)$. By finite-dimensionality of the motive of $X$, this
		decomposition lifts to a decomposition of the Chow motive of $X$\,:
		\begin{equation}\label{eq:motive}
		\h(X) \simeq \h^n(X) \oplus \bigoplus \mathds{1}(*),
		\end{equation}
		where $\mathrm{H}^*(\h^n(X)) = \mathrm{H}^n(X)$.
		
		By Proposition \ref{prop:Disting}, it is enough to show that $R^*(X) \subseteq
		\operatorname{DCH}^*_\phi(X)$. Since we already know that
		$\operatorname{DCH}^*_\phi(X)$ is a subalgebra of $\CH^*(X)$ that contains the
		Chern classes of $X$, it is enough to show that $\operatorname{DCH}^1_\phi(X)
		=
		\CH^1(X)$ and that the intersection of any two cycles of positive codimension
		belongs to $\operatorname{DCH}^*_\phi(X)$.  That $\operatorname{DCH}^1_\phi(X)
		=
		\CH^1(X)$ is clear from the description of the motive of $X$ given in
		\eqref{eq:motive}. By Proposition \ref{prop:multmarking}, $X$ has a
		multiplicative self-dual Chow--K\"unneth decomposition $\{\pi^i_X : 0\leq i
		\leq
		2n\}$ that induces a bigrading on the Chow ring of $X$ with the property that
		$\operatorname{DCH}_{\phi}^*(X) \subseteq \CH^*(X)_{(0)}$. By
		finite-dimensionality of $X$, any two Chow--K\"unneth decompositions of the
		motive of $X$ are isomorphic\,; therefore, by \eqref{eq:motive}, the bigrading
		on $\CH^*(X)$ has the form\,:
		\begin{equation}\label{eq:multX}
		\CH^i(X) = \CH^i(X)_{(0)} \oplus \CH^i(X)_{(2i-n)}.
		\end{equation}
		By multiplicativity we have $\operatorname{Im} \left( \CH^i(X)_{(s)} \otimes
		\CH^j(X)_{(t)} \to \CH^{i+j}(X)\right) \subseteq \CH^{i+j}(X)_{(s+t)}$.
		\linebreak
		Combining the multiplicativity with \eqref{eq:multX},
		we find that, for $0<i,j<n$, we have
		\begin{small}
			$$\operatorname{Im} \left( \CH^i(X) \otimes \CH^j(X) \to \CH^{i+j}(X)\right)
			=
			\operatorname{Im} \left( \CH^i(X)_{(0)} \otimes \CH^j(X)_{0} \to
			\CH^{i+j}(X)\right)  \subseteq \CH^{i+j}(X)_{(0)}.$$ 
		\end{small}
		\noindent Now the motive $(X,\pi_X^{2i})$ is isomorphic to a direct sum of
		Lefschetz motives $\mathds{1}(-i)$ for all  $i\neq \frac{n}{2}$, so that
		$\CH^i(X)_{(0)}$ maps isomorphically to $\overline{\CH}^i(X)$ for all  $i\neq
		\frac{n}{2}$. Since $\operatorname{DCH}_\phi^i(X)$ maps isomorphically to
		$\overline{\CH}^i(X)$ for all  $i$,  we see that the inclusion
		$\operatorname{DCH}_\phi^i(X) \subseteq \CH^i(X)_{(0)}$ is an equality for all
		$i\neq \frac{n}{2}$. This establishes that the intersection of any two cycles
		of
		positive codimension belongs to $\operatorname{DCH}^*_\phi(X)$, and thereby
		finishes the proof of Proposition \ref{prop:crucial}.
	\end{proof}

	\begin{remark}\label{rmk:reduction}
		In view of Proposition~\ref{prop:Disting} (together with the fact that the
		property $(\star)$ is stable under product by \cite[Prop.~4.1]{FV}),
		Propositions~\ref{prop:multmarking} and
		\ref{prop:crucial}, the proof of the main Theorems \ref{t:bv} and \ref{t:mck}
		reduces to showing that the relevant varieties $X$ satisfy the
		following two
		properties\,:
		\begin{enumerate}
			\item The cohomology of $X$ is spanned by algebraic classes in degree $\neq
			n$\,;
			\item $X$ admits a marking that satisfies the condition $(\star)$.
		\end{enumerate}
	\end{remark}

	\section{The motive of $\bar{X}$}

	\subsection{Hyperelliptic curves} Let $C$ be a smooth projective hyperelliptic
	curve of genus $g\geq 0$, that is, $C$ comes equipped with a $2$-to-$1$
	morphism
	$\pi: C\to \mathbb{P}^1$. This morphism induces an involution on $C$ which we
	call the hyperelliptic involution. By definition, the \emph{Weierstra\ss\
		points} of $C$ are the $2g+2$ ramification points of the morphism $\pi:C\to
	\mathbb{P}^1$, that is, the $2g+2$ fixed points of the involution. An elliptic
	curve will be seen as a hyperelliptic curve via its $[-1]$-involution. We have
	the following basic lemma.
	
	\begin{lemma}\label{lem:fixedpointshyperelliptic}
		The fixed points for the hyperelliptic involution are pairwise rationally
		equivalent, \emph{i.e.}, define the same class in $\CH_0(C)_\QQ$.
	\end{lemma}
	\begin{proof}
		Since $\pi$ is flat of degree $2$, we see that any two Weierstra\ss\ points
		$P$
		and $Q$ of $C$ satisfy $2[P] = 2[Q] \in \CH_0(C)$. Therefore the Weierstra\ss\
		points on $C$  define the same class in $\CH_0(C)_\QQ$.
	\end{proof}
	
	\subsection{The hyperelliptic curves $y^2 = x^{2g+1} + D$} \label{sec:Cg}
	Let $g$ be a natural number and let $D$ be a non-zero rational number, and 
	let $C_{g,D}$ be the smooth projective model of the affine curve $Y = \{y^2 =
	x^{2g+1} + D\}$. When $g>1$, the projective closure $X$ of $Y$ has a cusp at
	the
	point $\infty$, and the hyperelliptic curve $C_{g,D}$ is its normalization. In
	particular $C_{g,D}$ is obtained from $Y$ by adding one additional point at
	$\infty$.
	
	The curve $C_{g,D}$ is endowed with the hyperelliptic involution $\sigma$ which
	on the open subset $Y$ is given by $(x,y) \mapsto (x,-y)$. The fixed points for
	that
	action are called the Weierstra\ss\ points, and are explicitly given by the
	$2g+2$ points $\{(\zeta\cdot |D|^{\frac{1}{2g+1}}, 0) : \zeta \in \mu_{2g+1}\}
	\cup \{\infty\}$.
	
	The curve is also endowed with an action of $\mu_{2g+1}$, which on the open
	subset $Y$
	is given by $\zeta \cdot (x,y) = (\zeta \cdot x,y)$. Its fixed points are the
	points $(0,\sqrt{D})$, $(0,-\sqrt{D})$, and $\infty$. Note that these points
	are
	defined over $\QQ$ if $D=1$.
	
	\begin{lemma}\label{lem:fixedpoints}
		The fixed points for the hyperelliptic involution and for the
		$\mu_{2g+1}$-action are pairwise rationally equivalent, \emph{i.e.}, define
		the same
		class in $\CH_0(C_{g,D})_\QQ$.
	\end{lemma}
	\begin{proof}
		By Lemma \ref{lem:fixedpointshyperelliptic}, the  fixed points for the
		hyperelliptic involution on $C_{g,D}$ (which include the $\infty$ point)
		define
		the same class in $\CH_0(C_{g,D})_\QQ$.
		
		Consider the lines $\{y=\sqrt D\}, \{y=-\sqrt D\}$ and $\{x=0\}$ in
		$\mathbb{A}^2$. These lines intersect the curve $Y$ in $(2g+1)[(0,\sqrt D)]$,
		$(2g+1)[(0,-\sqrt D)]$ and $[(0,\sqrt D)] + [(0,-\sqrt D)]$, respectively. We
		deduce that these 3 cycles are rationally equivalent on $Y$. By excision, we
		see
		that the points $(0,\sqrt D)$, $(0,-\sqrt D)$ and $\infty$ define the same
		class
		in $\CH_0(C_{g,D})_\QQ$. (Note that in the case $g=1$, \emph{i.e.}, in the
		case where
		$(C_{g,D},\infty)$ is an elliptic curve, the points $(0,\sqrt D)$ and 
		$(0,-\sqrt D)$ are 3-torsion points.)
	\end{proof}

	\subsection{Key proposition}
	
	\begin{proposition}\label{prop:markinghyperelliptic}
		Let $C$ be a smooth projective curve equipped with the action of a finite
		group
		$H$. Assume that $(C,H)$ is either\,:
		\begin{enumerate}[(i)]
			\item a hyperelliptic curve equipped with its hyperelliptic involution\,;
			\item a hyperelliptic curve $C_{g,D}$ as in subsection \ref{sec:Cg} equipped
			with the action of $\mu_{2g+1}$.
		\end{enumerate}
		Then $C$ has a marking that satisfies the condition $(\star)$ and $(\star_H)$,
		with the additional property that if $P$ is a fixed point of $H$, then the
		embedding $P\hookrightarrow C$ is distinguished.
	\end{proposition}
	\begin{proof}
		By \cite[Corollary~5.4]{FV}, the embedding of a hyperelliptic curve inside its
		Jacobian $\mathrm{AJ} : C\to J(C), x \mapsto O_C(x-q)$, where $q$ is a
		Weierstra\ss\ point, provides a marking for $C$ that satisfies $(\star)$.
		Moreover the embedding $q \hookrightarrow C$ is distinguished by construction.
		Since by Lemma \ref{lem:fixedpoints} all fixed points of $H$ and all
		Weierstra\ss\ points are rationally equivalent, we see that the embedding
		$P\hookrightarrow C$ is distinguished for any choice of fixed point $P$ of
		$H$.
		
		It remains to show that for any $h\in H$, the graph $\Gamma_h \in
		\CH^1(C\times
		C)$ is distinguished with respect to the product marking $\mathrm{AJ}\otimes
		\mathrm{AJ}$. Let $P$ be a fixed point of $H$ (which by Lemma
		\ref{lem:fixedpoints} is rationally equivalent to any Weierstra\ss\ point) and
		consider the following Chow--K\"unneth decomposition 
		$$\pi_C^0 = P \times C, \quad \pi_C^2 = C\times P, \quad \mbox{and }\ \pi_C^1
		=
		\Delta_C - \pi^0_C - \pi^2_C.$$
		These are distinguished cycles in $C\times C$ and by Proposition
		\ref{prop:multmarking} they define a multiplicative Chow--K\"unneth
		decomposition such that 	$\operatorname{DCH}^*(C\times C) \subseteq
		\CH^*(C\times C)_{(0)}$. Since in codimension $1$ the previous inclusion is an
		equality (see Proposition \ref{prop:multmarking}), we are reduced to showing
		that $\Gamma_h$ belongs to $\CH^1(C\times C)_{(0)}$ with respect to the
		product
		Chow--K\"unneth decomposition on the product $C\times C$. That is, we are
		reduced to showing that $$(\pi_C^0\otimes \pi_C^1 + \pi_C^1\otimes \pi_C^0)_*
		\Gamma_h = 0 \quad \mbox{in } \CH^1(C\times C).$$
		By orthogonality and symmetry, we are reduced to showing that 
		$$\pi_C^1\circ \Gamma_h \circ \pi_C^2 = 0 \quad \mbox{in } \CH^1(C\times C).$$
		But $\Gamma_h \circ \pi_C^2 = C\times h(P) = C\times P = \pi_C^2$ (since $P$
		is a fixed point of $H$). 
		We can then conclude by orthogonality of $\pi_C^1$ and
		$\pi_C^2$.
	\end{proof}

	\subsection{The motive of $\bar X$} 
	\label{sbar}     
	In this subsection we consider a projective variety of the form 
	$$\bar X = (C_1\times \cdots \times C_n)/G,$$
	where the $C_i$ are hyperelliptic curves and $G$ is a certain finite subgroup
	of
	automorphisms of $C_1\times \cdots \times C_n$. Specifically, we assume one of
	the following\,:
	\begin{enumerate}[(a)]
		\item Each $C_i$ is a hyperelliptic curve equipped with the action of $H\cong
		\ZZ_{2}$ induced by its hyperelliptic involution, and $$G = \{(h_1,\ldots,h_n)
		\in H^n : h_1 +\cdots+h_n = 0\}.$$
		\item Let $g$ be a positive integer and let $a,b\geq 0$ be integers such that
		$n=a+b$ and $a>b$. Each $C_i$ is a hyperelliptic curve of genus $g$ as in
		Subsection
		\ref{sec:Cg} equipped with the action of $H = \mu_{2g+1}$ given by $(x,y)
		\mapsto (\zeta\cdot x,y)$, and 
		$$G = G_{a,b} =  \{(h_1,\ldots,h_n) \in H^n : h_1 +\cdots+h_a - h_{a+1} -
		\cdots - h_{a+b} = 0\}.$$
	\end{enumerate}	
	Note that in case (a) if the curves are chosen to be elliptic curves endowed
	with the $[-1]$-involution, then $\bar{X}$ is the variety considered in Theorem
	\ref{ch}, while   in case (b) if the curves are chosen to be elliptic curves
	with an order 3 automorphism, and one takes $b=0$, then $\bar{X}$ is the
	variety
	considered in Theorem \ref{ch3}.	\medskip
	
	The goal of this subsection is to determine the motive of $\bar{X}$\,; this
	will be used
	later on in Section \ref{s:main}. We also show
	how the formalism of distinguished cycles (and multiplicative Chow--K\"unneth
	decomposition) works for
	$\bar{X}$. This is done for the reader's benefit, and is not necessary for the
	results in Section \ref{s:main}.
	\medskip

	In what follows, a hyperelliptic curve $C$ is always endowed with the
	Chow--K\"unneth decomposition given by 
	\[\pi_C^0 := P\times C, \quad \pi_C^2 = C\times P, \quad \pi^1_C = \Delta_C
	- \pi_C^0 - \pi_C^2,\]
	where $P$ is the class of a Weierstra\ss\ point.
	A product of hyperelliptic curves $C_1\times\cdots\times C_n$ is endowed with
	the product Chow--K\"unneth
	decomposition
	$$\pi^k :=\sum_{k=k_1+\cdots +k_n} \pi_{C_1}^{k_1}\otimes\cdots\otimes
	\pi_{C_n}^{k_n}.$$
	In the case where $C$ is an elliptic curve endowed with the $[-1]$-involution,
	note that the $0$ element is a Weierstra\ss\ point, so that the above
	Chow--K\"unneth decomposition is the Deninger--Murre decomposition \cite{DM}.
	By unicity of the Deninger--Murre decomposition \cite[Theorem
	3.1]{DM}, the above product Chow--K\"unneth decomposition for a product of
	elliptic curves is the one of Deninger--Murre
	\cite{DM}.
	\medskip

	Since the variety $\bar X$ is obtained as the quotient of the smooth projective
	variety $C_1\times \cdots \times C_n$ by a
	finite group $G$, and since we are only concerned with algebraic cycles with
	rational coefficients, the motive of $\bar X$ identifies with the $G$-invariant
	part of the motive of $C_1\times \cdots \times C_n$, as algebra objects. In
	particular, the notion of
	multiplicative Chow--K\"unneth decomposition and the condition
	$(\star_{\mathrm{Mult}})$ make sense for $\bar X$, so that Proposition
	\ref{prop:crucial} holds with the Chern classes omitted. These are established
	for $\bar X$ via the
	following proposition, which is the main result of this section\,:

	\begin{proposition}\label{prop:Xbar}
		Let $\bar{X}=(C_1\times\cdots\times C_n)/G$ with $C_1,\ldots,C_n$ and $G$ as
		in (a) or (b) above. Then the $\QQ$-subalgebra of $\CH^*(\bar X)$ generated by
		$\CH^1(\bar
		X)$ and by the images of the intersection products
		\[ \CH^{k-l}(\bar{X})\otimes \CH^l(\bar{X})\ \rightarrow\ \CH^{k}(\bar{X})\ \
		\
		(0< l<k) \]
		injects into $\overline\CH^{*}(\bar{X})$.
	\end{proposition}

	In the next section, this statement, together with the existence of a
	multiplicative Chow--K\"unneth decomposition, will be extended from $\bar{X}$
	to
	the crepant resolution $X$ for Calabi--Yau varieties as in Theorems
	\ref{ch} and \ref{ch3}, and to the Schreieder resolution $X$ for varieties as
	in Theorem~\ref{schreieder}.
	\medskip

	Note that  Proposition \ref{prop:markinghyperelliptic} (together with
	\cite[Remark 4.3 and Proposition 4.12]{FV}) establishes
	the existence of a marking for $\bar X$ that satisfies
	$(\star_{\mathrm{Mult}})$.	
	Therefore the proof of Proposition \ref{prop:Xbar} reduces, thanks to
	Proposition \ref{prop:crucial} and Remark \ref{rmk:reduction}, to an explicit
	computation of
	the Chow motive of $\bar X$. This is achieved in Corollary \ref{C:CH}. These
	computations will also be used in Section~\ref{svois} in order to establish
	Theorem \ref{conjV}. First, we start with a general lemma and a general
	proposition.

	\begin{lemma}\label{L:aut} Let $C$ be a smooth projective curve endowed with
		the action of a finite group $H$ such that $C/H$ is rational. Then, choosing a
		degree-1 zero-cycle $\alpha$ on $C$ that is $H$-invariant (e.g. $\alpha =
		\frac{1}{|H|}\sum_{h\in H} h^* [P]$ for any choice of point $P\in C$), and
		denoting $\pi^0_C := \alpha \times C$, $\pi^2_C := C\times \alpha$ and
		$\pi_C^1:= \Delta_C - \pi^0_C - \pi^2_C$, we have 
		\begin{equation}\label{eq:1}
		\sum_{h\in H} \Gamma_h\circ \pi_{C}^{1}=    0\ \ \ \hbox{in}\
		\CH^1(C\times C),
		\end{equation}
		whereas
		\begin{equation}\label{eq:2}
		\Gamma_h\circ \pi_{C}^{j} = \pi_{C}^{j}\ \ \ \hbox{in}\ \CH^1(C\times
		C), \  \ \ \mbox{for} \ j = 0 \ \mbox{or}\ 2, \ \mbox{and for all} \ h\in H.
		\end{equation}
		In particular, if $E$ is an elliptic curve and $H$ is a  non-trivial subgroup
		of the group of automorphisms, then $\sum_{h\in H} \Gamma_h\circ \pi_E^1 = 0,$
		where $\pi_E^1$ is the Chow--K\"unneth projector of Deninger--Murre.
	\end{lemma}
	\begin{proof} That $\Gamma_h\circ\pi_C^0 = \pi_C^0$ and $\Gamma_h\circ\pi_C^2 =
		\pi_C^2$ for
		all $h\in H$ is clear.
		Let $p: C\to C/H$ be the projection morphism. On the one hand, we have
		$${}^t\Gamma_p \circ \Gamma_p = \sum_{h\in H} \Gamma_h.$$
		On the other hand, since $C/H$ is rational we have $\Delta_{C/H} = \pi_{C/H}^0
		+\pi_{C/H}^2$ with $\pi_{C/H}^0 = \beta \times C/H$ and $\pi_{C/H}^2 = C/H
		\times \beta$ for any choice of degree-1 zero cycle $\beta$ on $C/H$. We also
		have 
		\[ 	   {}^t\Gamma_p \circ \Gamma_p  = {}^t\Gamma_p \circ (\pi_{C/H}^0
		+\pi_{C/H}^2) \circ \Gamma_p 
		=   (p\times p)^\ast (\pi^0_{C/H} + \pi^2_{C/H}) = \vert H\vert (\pi^0_{C} +
		\pi^2_{C}).
		\]
		We conclude by orthogonality of the Chow--K\"unneth projectors $\pi_C^i$.
	\end{proof}

	\begin{proposition}\label{P:quotient}
		Let $C_1,\ldots, C_n$ be smooth projective curves endowed with the action of
		finite abelian group $H$ such that each $C_i/H$ is rational. For integers
		$a,b\geq 0$ such that $a+b = n$ and $a>b$, consider the group
		$$G = G_{a,b} =  \{  (h_1,\ldots,h_n)\ \in H^n\ : \ h_1+\cdots+h_a - h_{a+1} -
		\cdots - h_{a+b}=0_H \}$$ together with
		its natural action on the product $C_1\times\cdots \times C_n$ and with the
		induced quotient morphism $p : C_1\times \cdots \times C_n \to (C_1\times
		\cdots
		\times C_n)/G$.  Then we have the implication
		\[ 0 < |\{j : i_j =1 \} |<n  \quad \Rightarrow \quad \Gamma_p\circ
		(\pi_{C_1}^{i_1} \otimes \cdots \otimes \pi_{C_n}^{i_n}) \circ {}^t\Gamma_p 
		=0.
		\]
		In particular, the Chow motive of $(C_1\times \cdots \times C_n)/G$ decomposes
		into a direct sum of Lefschetz motives and one copy of the motive $T:=
		(C_1\times \cdots \times C_n, \frac{1}{|G|} \sum_{g\in G} \Gamma_g\circ 
		(\pi_{C_1}^{1} \otimes \cdots \otimes \pi_{C_n}^{1}))$. 
	\end{proposition}
	\begin{proof}
		Let us write $\Pi$ for $\pi_{C_1}^{i_1} \otimes \cdots \otimes
		\pi_{C_n}^{i_n}$. The action of $G$ commutes with $\Pi$, therefore
		$\frac{1}{|G|}\Gamma_p \circ \Pi \circ {}^t\Gamma_p$ is an idempotent, and it
		is
		zero if and only if $$\sum_{g\in G} \Gamma_g \circ \Pi = 0.$$
		Assume that $0 < |\{j : i_j =1 \} |<n$. By symmetry, we may assume without
		loss
		of generality that $\Pi = \pi_{C_1}^1\otimes \Pi' \otimes \pi_{C_n}$, where
		$\pi_{C_n} = \pi_{C_n}^0$ or $ \pi_{C_n}^2$, and $\Pi' = \pi_{C_2}^{i_2}
		\otimes
		\cdots \otimes \pi_{C_{n-1}}^{i_{n-1}}$. Then, partitioning $G$ by the first
		entry of its elements, we have
		\[\sum_{g\in G} \Gamma_g\circ \Pi = \sum_{g' := (h_2,\ldots, h_{n-1}) \in
			H^{n-2}}  
		\left( \sum_{h\in H} (\Gamma_h\circ \pi_{C_1}^1)\otimes (\Gamma_g'\circ \Pi')
		\otimes
		\pi_{C_n} \right)=   0.
		\]
		The first equality follows from \eqref{eq:1} and the second equality follows
		from \eqref{eq:2} of Lemma \ref{L:aut}.

		Now  assume $|\{j : i_j =1 \} | = 0$, \emph{i.e.}, $\Pi = \pi_{C_1}^{i_1}
		\otimes \cdots \otimes \pi_{C_n}^{i_n}$ with $\{i_1,\ldots, i_n\} \subseteq
		\{0,2\}$. In that case, we also have for all $g\in G$ that $\Gamma_g\circ \Pi
		= \Pi$,
		and thus $\sum_{g\in G} \Gamma_g \circ \Pi \neq 0$. Moreover the motive
		$((C_1\times
		\cdots \times C_n)/G, \frac{1}{G} \Gamma_p\circ \Pi \circ {}^t\Gamma_p)$ is
		isomorphic to the Lefschetz motive $\mathds{1}(-k)$, where 
		$k = |\{j : i_j = 2\}|$.
		
		Finally, when considering $|\{j : i_j =1 \} | = n$, one is left with the
		motive
		\[ ((C_1\times\cdots\times C_n)/G, \frac{1}{|G|}\Gamma_p \circ
		(\pi^1_{C_1}\times\cdots\times \pi^1_{C_n}) \circ {}^t\Gamma_p	)\ ,\]
		which (under $\Gamma_p$) is isomorphic to	
		\[ T:=
		(C_1\times \cdots \times C_n, \frac{1}{|G|} \sum_{g\in G} \Gamma_g\circ 
		(\pi_{C_1}^{1} \otimes \cdots \otimes \pi_{C_n}^{1}))\ .\]
	\end{proof}

	\begin{corollary}\label{C:CH}
		Let $\bar{X}=(C_1\times\cdots\times C_n)/G$ with $C_1,\ldots,C_n$ and $G$ as in
		(a) or (b) above.
		Then the Chow motive of $\bar{X}$
		decomposes into a
		direct sum of Lefschetz motives and one copy of the motive 
		\[ T:= (C_1\times
		\cdots \times C_n, \frac{1}{|G|} \sum_{g\in G} \Gamma_g\circ  (\pi_{C_1}^{1}
		\otimes \cdots \otimes \pi_{C_n}^{1}))\ \ \ \in \MM_{\rm rat}\ .\] 
		In case (a),
		$T:=(C_1\times \cdots \times C_n, \pi_{C_1}^{1} \otimes \cdots \otimes
		\pi_{C_n}^{1})$, and in case (b) the motive $T$ is such that $H^j(T)=0$ for
		$j\not=n$, and its transcendental cohomology has Hodge numbers	
		\[ h_{tr}^{p,n-p}(T)=\begin{cases}  g &\hbox{for\ }p\in\{a,b\}\,;\\
		0 &\hbox{for\ } p\not= {n\over 2}\ .\\
		\end{cases}\]
	\end{corollary}
	\begin{proof}
		By Proposition \ref{P:quotient}, we only need to compute $T$.
		
		(a)	Denote $\sigma_i$ the non-trivial hyperelliptic involution of $C_i$. By
		Lemma \ref{L:aut}, we have 
		$$\sigma_i \circ \pi_{C_i}^1 = -\pi_{C_i}^1.$$
		Writing $\Pi = \pi_{C_1}^{1}
		\otimes \cdots \otimes \pi_{C_n}^{1}$, we therefore have for all $g\in G$
		that $g\circ \Pi = \Pi$, and thus $\sum_{g\in G} g \circ \Pi = |G| \Pi \neq
		0$.
		In particular, we see that the motive $((C_1\times \cdots \times C_n)/G,
		\frac{1}{|G|} \Gamma_p\circ (\pi_{C_1}^{1} \otimes \cdots \otimes
		\pi_{C_n}^{1})\circ {}^t\Gamma_p)$ is isomorphic to the motive $(C_1\times
		\cdots \times C_n, \pi_{C_1}^{1} \otimes \cdots \otimes \pi_{C_n}^{1})$.
		
		(b)	 In case $g=1$ and $b=0$ (which is the set-up of Theorem \ref{ch3}), it
		is proven in \cite[Theorem 3.3]{CH} that $H^n(T)$ has
		dimension 2 when $n$ is odd,
		and that $H^n_{tr}(T)$ has
		dimension $2$ when $n$ is even. 
		
		For the case $g>1$, this is essentially done in Schreieder \cite{Schreieder}.
		The proof goes as follows in the case where each curve $C_i$ is the curve
		$C_{g,1}$. A basis of $H^{1,0}(C_{g,1})$ is given by the differential forms
		$\omega_i = \frac{x^{i-1}}{y}\mathrm{d}x$, $1\leq i \leq g$, and for $\zeta \in
		\mu_{2g+1}$ we have $\zeta^*\omega_i = \zeta^i \omega_i$. The proof then
		consists in understanding the invariants in $H^{0,1}(C_{g,1})^{\otimes d}
		\otimes H^{1,0}(C_{g,1})^{\otimes (n-d)}$. Since this result will not be used in
		this paper, let us only mention that this can be done combinatorially and was
		essentially carried out by Schreieder in \cite[Lemma~8]{Schreieder}. 
	\end{proof}

	\begin{proof}[Proof of Proposition \ref{prop:Xbar}]
		In view of Proposition \ref{prop:markinghyperelliptic} and Corollary
		\ref{C:CH}, this
		is an immediate consequence of Proposition \ref{prop:crucial} (with the Chern
		classes omitted).
	\end{proof}

	\section{The motive of $X$}
	\label{s:main}
	
	This section contains the proof of the main result of this note\,:	
	
	\begin{theorem}\label{t:bv}
		Let $X$ be either a Calabi--Yau variety as in Theorem \ref{ch} or \ref{ch3},
		or
		the Schreieder variety of Theorem \ref{schreieder}.
		Then for all integers $m\geq 1$ the $\QQ$-algebra epimorphisms $\CH^*(X^m) \to
		\overline{\CH}^*(X^m)$ admit a
		section whose image contains the (Chow-theoretic) Chern classes of $X^m$. 
		Moreover, assuming $\dim X\geq 2$, the graded subalgebra $R^*(X) \subseteq
		\CH^*(X)$ generated by
		divisors, Chern classes and by cycles that are the intersection of two cycles
		in~$X$ of positive codimension  injects into $\overline{\CH}^*(X)$.
	\end{theorem}
	
	We also establish the following\,:
	
	\begin{theorem}\label{t:mck}Let $X$ be either a Cynk--Hulek
		Calabi--Yau variety of dimension $n$ as in Theorem \ref{ch} or \ref{ch3}, or a
		Schreieder variety as in Theorem \ref{schreieder}. 
		Then $X$ admits a multiplicative self-dual Chow--K\"unneth decomposition, in
		the sense of \cite{SV}.
	\end{theorem}

	Before giving the proof of Theorems \ref{t:bv} and \ref{t:mck}, we detail
	the inductive constructions of Cynk--Hulek \cite{CH} and Schreieder
	\cite{Schreieder}. This will allow us to prove the theorems, by applying the
	reduction argument outlined in Remark \ref{rmk:reduction}. That is, the proof
	will consist
	in checking
	that each step of the construction only changes algebraic classes in
	cohomology, and
	preserves the condition $(\star)$ of Definition \ref{def:Star}.

	\subsection{$\ZZ_{2}$-actions}
	\label{ss:z2}
	
	\begin{proposition}\label{p:inductionZ2}
		Let $X_i$, $i=1,2$, be smooth
		projective varieties endowed with an action of $H_i = \ZZ_{2}$. Assume that,
		for
		$i=1,2$,
		$(X_i,H_i)$ enjoys the following properties\,:
		\begin{enumerate}[(i)]
			\item $X_i$ has a marking that satisfies $(\star)$ and $(\star_{H_i})$\,; 
			\item the quotients $X_i/H_i$ are smooth\,; 
			\item $B_i := \mathrm{Fix}_{X_i}(H_i)$ is a smooth divisor\,;
			\item  $B_i$ has trivial Chow groups 
			(in particular, $B_i$ has a marking that satisfies $(\star)$)\,;	
			\item The inclusion morphism $B_i \hookrightarrow X_i$ is distinguished with
			respect to the above markings.
		\end{enumerate}
		Let $Z$ be the blow-up of  $X_1\times X_2$ along $B_1\times B_2$, and let
		$\tilde{B}$ be the exceptional divisor\,; the
		action of $H_1\times H_2$ on $X_1\times X_2$ naturally endows $Z$ with an
		action
		of $H_1\times H_2$. Let $$G:= \{(h_1,h_2)\in H_1\times H_2 : h_1+h_2 = 0\}.$$
		Then the
		quotient variety $X:=Z/G$ is smooth and  the pair $(X,H):=(Z/G,(H_1\times
		H_2)/G)$ enjoys
		properties (i)--(v), with $\tilde{B} = \mathrm{Fix}_{X}(H)$.
	\end{proposition}
	
	\begin{proof} This is our take on the inductive construction of Cynk--Hulek
		\cite[Propositions 2.1 and 2.2]{CH} (where the $X_i$ are in addition assumed
		to
		be Calabi--Yau, and it is proven that the resulting variety $X:=Z/G$ is again
		Calabi--Yau). As in \emph{loc. cit.}, the various varieties fit into a
		commutative diagram
		\[ \xymatrix{
			Z \ar[r] \ar[d] &  X_1\times X_2 \ar[d]\\
			X:=Z/G\ar[r] \ar[d] & (X_1\times X_2)/ G \ar[d]\\
			Y:=Z/(H_1\times H_2) \ar[r]   \ \  &  \ \ \ X_1/H_1\times
			X_2/H_2=:Y_1\times Y_2
		}\]
		(where we adhere to the notation of \cite{CH}). Here, horizontal arrows
		are blow-ups, and vertical arrows are 2-to-1 morphisms. This (doubly !)
		explains
		why $X$ is smooth. On the one hand, $X$ is the blow-up of the quotient
		$(X_1\times X_2)/G$ along the singular locus (which is isomorphic to
		$B_1\times
		B_2$) consisting of $A_1$ singularities.
		On the other hand, $X$ is the double
		cover of the smooth variety $Y$ branched along the smooth divisor obtained by
		blowing up the smooth image of $B_1\times B_2$ in $Y_1\times Y_2$. This also
		shows that the fixed loci $\mathrm{Fix}_Z(H_1\times H_2)$ and
		$\mathrm{Fix}_Z(G)$	(which coincide with the branch loci of the covers $Z\to
		Y$,
		resp. $Z\to X$) are isomorphic to the exceptional divisor $\tilde{B}$
		\cite[Proof of Proposition 2.1]{CH}.
		
		Let us endow $X_1\times X_2$ and $B_1\times B_2$ with the product markings\,;
		these
		satisfy $(\star)$ by \cite[Prop.~4.1]{FV}, the inclusion morphism $B_1\times
		B_2
		\hookrightarrow X_1\times X_2$ is distinguished by \cite[Prop.~3.5]{FV}, and
		the
		pushforwards and pullbacks along the  projection morphisms $X_1\times X_2 \to
		X_i$ and $B_1\times B_2 \to B_i$ are distinguished. Moreover, the pair
		$(X_1\times X_2, H_1\times H_2)$, where $\mathrm{Fix}_{X_1\times
			X_2}(H_1\times H_2) = B_1\times
		B_2$, satisfies properties
		(i) and (v) by \cite[Proposition 4.1]{FV} (and also \cite[Remark~4.3]{FV}).
		Since $\tilde{B}$ is a $\PP^1$-bundle over $B_1\times B_2$, property (iv) is
		satisfied.
		
		Since $\tilde{B} = \mathrm{Fix}_Z(H_1\times H_2) = \mathrm{Fix}_Z(G)$,
		it
		follows 
		that $(Z,H_1\times H_2)$ satisfies $(iii)$.
		Now this is
		enough to ensure that $(Z,H_1\times H_2)$ satisfies properties (i)--(v)
		by \cite[Proposition 4.8]{FV}. 
		
		Consider now the quotient morphism $Z\to X=Z/G$\,; it is a
		$\ZZ_2$-covering branched along the smooth divisor $\tilde{B}$ (which we view
		as a divisor on $Z$ and $X$ via the quotient morphism). 
		We have already seen that
		$\tilde B$ satisfies $(\star)$\,; and $X$ satisfies $(\star)$ by
		\cite[Prop.~4.12]{FV}. That the inclusion morphism $\tilde{B} \to X$ is
		distinguished
		follows from the fact that the inclusion morphism $\tilde B \to Z$ is
		distinguished and the fact that the quotient morphism $Z\to X$ is
		distinguished
		[\emph{loc. cit.}]. In order to conclude, it remains to see that $X$ satisfies
		$(\star_{H})$. But then this again follows from the fact that the
		quotient
		morphism $Z\to X$ is distinguished, together with the fact that $Z$ satisfies
		$(\star_{H_1\times H_2})$.
	\end{proof}

	\subsection{$\ZZ_3$-actions}
	\label{sec:3} 
	We take care of the inductive approach in order to
	treat the case of the Cynk--Hulek Calabi--Yau varieties of Theorem \ref{ch3}.
	This is very similar to the arguments in the next subsection, but we include
	detailed
	arguments here for the sake of readability.

	\begin{proposition}\label{p:inductionZ3}
		Let $X_i$, $i=1,2$, be smooth
		projective Calabi--Yau varieties endowed with an action of $H_i = \ZZ_{3}$.
		Assume the following properties\,:
		\begin{enumerate}[(i)]
			\item The action of $H_i$ on $X_i$ does not preserve the canonical form of
			$X_i$\,;
			\item $X_i$ has a marking that satisfies $(\star)$ and $(\star_{H_i})$\,; 
			\item $B_1 := \mathrm{Fix}_{X_1}(H_1)$ is a smooth divisor, whereas  $B_2 :=
			\mathrm{Fix}_{X_2}(H_2)$ is the disjoint union of a 
			smooth divisor $B_{2,1}$ and a smooth codimension $2$ subvariety $B_{2,2}$\,;
			
			\item $B_i$ has trivial Chow groups (in particular, $B_i$ has a marking that
			satisfies $(\star)$)\,; 
			\item The inclusion morphism $B_i \hookrightarrow X_i$ is distinguished with
			respect to the above markings.
		\end{enumerate}
		Let $$G:= \{(h_1,h_2)\in H_1\times H_2 : h_1+h_2 = 0\}\ .$$		 
		Then there exists a crepant resolution of singularities
		\[ X\ \to\ (X_1\times X_2)/G\ ,\]	
		and an action of $H=(H_1\times H_2)/G\cong\ZZ_3$ on $X$ (induced by the action
		of $\ide\times
		H_2$ on $X_1\times X_2$), such that the pair
		$(X,H)$ satisfies the same assumptions as $(X_2,H_2)$.
	\end{proposition}

	\begin{proof} This is essentially the inductive argument of \cite[Proposition
		3.1]{CH}, on which we have additionally grafted condition $(\star)$.
		We briefly resume the construction of $X$ given in \cite[Proposition 3.1]{CH}
		(retaining the notation of \emph{loc. cit.}).
		
		The quotient $(X_1\times X_2)/G$
		has $A_2$-singularities along a codimension $2$ stratum $W_1$ (isomorphic to
		$B_1\times B_{2,1}$), plus other singularities along a codimension $3$ stratum
		$W_2$ (isomorphic to $B_1\times B_{2,2}$). A crepant resolution
		\[ X\ \to\ (X_1\times X_2)/G\ \]	
		is explicitly described in local coordinates in \cite[Proof of Proposition
		3.1]{CH}.  Moreover, it is checked in \cite[Proof of Proposition 3.1]{CH} 
		that $(X,H)$ satisfies conditions (i) and (iii) (just as $(X_2,H_2)$).
		Therefore it only remains to check that $X$ also satisfies
		conditions (ii), (iv) and (v).

		As explained in \emph{loc. cit.}, the variety $X$ can also be obtained as
		follows\,:
		Let $Z_1$ be the blow-up of  $X_1\times X_2$ along $B_1\times B_{2}$.		
		The action of $$G:= \{(h_1,h_2)\in H_1\times H_2 : h_1+h_2 = 0\}.$$		 
		on $X_1\times X_2$ naturally endows $Z$ with an action
		of $G$. Let $Z_2\to Z_1$ be the blow-up with center the codimension $2$ part
		of $\mathrm{Fix}_{Z_1}(G)$ (this center consists of two disjoint copies of
		$W_1$, as can be seen from \cite[Lemma 18]{Schreieder}). The action of $G$
		lifts to $Z_2$, and we define
		\[ Z:=Z_2/G.\]
		The crepant resolution $X$ is now attained by performing a blow-down 
		\[ b : Z\ \to\ X, \]
		where the exceptional divisor of $b$ in $Z$ corresponds to the strict
		transform
		of the exceptional divisor of the first blow-up $Z_1\to X_1\times X_2$. The
		exceptional locus $V\subset X$ of $b$ is an isomorphic copy of $W_1$ (this
		exceptional locus $V\cong W_1$ corresponds to the intersection of the $2$
		irreducible components of the exceptional divisor in $X$ lying over the
		stratum
		$W_1$).
		
		Once again, we endow $X_1\times X_2$ and $B_1\times B_2$ with the product
		markings. These markings
		satisfy $(\star)$ by \cite[Prop.~4.1]{FV}, the inclusion morphism $B_1\times
		B_2
		\hookrightarrow X_1\times X_2$ is distinguished by \cite[Prop.~3.5]{FV}, and
		the
		pushforwards and pullbacks along the  projection morphisms $X_1\times X_2 \to
		X_i$ and $B_1\times B_2 \to B_i$ are distinguished. Moreover, the pair
		$(X_1\times X_2, H_1\times H_2)$, where $\mathrm{Fix}_{X_1\times
			X_2}(H_1\times H_2) = B_1\times
		B_2$, satisfies properties
		(ii), (iv) and (v) by \cite[Prop.~4.1]{FV} (and also \cite[Rem~4.3]{FV}, plus
		the fact that condition (iv) is stable under taking products). In view of
		\cite[Prop.~4.8]{FV}, this implies that
		$(Z_1,H_1\times H_2)$ and $(Z_1,G)$ satisfy (ii). 
		
		The codimension $2$ part of $\mathrm{Fix}_{Z_1}(G)$ consists of $2$ disjoint
		copies of $W_1\cong B_1\times B_{2,1}$, and so it has a marking satisfying
		$(\star)$. Let $E_1\subset Z_1$ denote the exceptional divisor. The inclusion
		of
		the $2$ copies of $W_1$ in $Z_1$ is distinguished, because the inclusion
		morphism $W_1\to E_1$  is distinguished (indeed, both $W_1$ and $E_1$ have
		trivial Chow groups), and the inclusion morphism $E_1\to Z_1$ is also
		distinguished. Again applying \cite[Prop.~4.8]{FV}, and reasoning as
		before, this implies that
		$(Z_2,H_1\times H_2)$ and $(Z_2,G)$ in turn satisfy (ii), (iv) and (v). 	
		
		The next step is to take the quotient $Z_2\to Z:=Z_2/G$. Here, 
		\cite[Prop.~4.12]{FV} ensures that $Z$ has a marking satisfying $(\star)$
		and that the quotient morphism $Z_2\to Z$ is distinguished. This last fact,
		combined with the fact that $(Z_2,H_1\times H_2)$ verifies (ii), ensures that
		$(Z,H)$ also verifies (ii).	
		The fact that $(Z_2,H_1\times H_2)$ verifies (v), plus the fact that the
		quotient morphism $Z_2\to Z$ is distinguished, ensures that $(Z,H)$ also
		verifies (v). Condition (iv) is satisfied for $(Z,H)$ since the fixed locus is
		dominated by the fixed locus of $(Z_2,G)$ which satisfied (iv).		
		
		The final step in the inductive process is the blow-down $b$ from $Z$ to $X$.
		Here, we know that the exceptional divisor $E\subset Z$ of $b$ has a marking
		that verifies $(\star)$ and is such that the inclusion is distinguished. Also,
		we
		know that the exceptional locus $V\subset X$ (is isomorphic to $W_1$ and so)
		has
		trivial Chow groups, and thus verifies $(\star)$. We remark that the
		correspondence 
		\[  {}^t \Gamma_b\circ \Gamma_b\ \ \in\  \CH^n(Z\times Z)\] 
		is supported on $\Delta_Z \cup E\times_V E$
		(by refined intersection). The fiber product $E\times_V E$ is a $\PP^1\times
		\PP^1$-bundle over $V$; as such, it is smooth irreducible of dimension $n$ and
		has trivial Chow groups. The inclusion $E\times_V E\subset E\times E$ is
		distinguished (both sides have trivial Chow groups), and the inclusion
		$E\times
		E\subset Z\times Z$ is distinguished (as $E\subset V$ is distinguished).
		Therefore, $E\times_V E\subset Z\times Z$ is distinguished, and so we may
		conclude that  ${}^t \Gamma_b\circ \Gamma_b$ is distinguished in
		$\CH^n(Z\times
		Z)$.
		
		Now, applying \cite[Prop.~4.9]{FV}, it follows
		that $X$ has a marking that verifies $(\star_{\mathrm{Mult}})$, and such that
		the blow-up morphism $b$ is distinguished. To show that the marking of $X$
		also
		verifies $(\star_{\mathrm{Chern}})$, one can reason as in the technical
		\cite[Lemma 6.4]{SV},
		with $\operatorname{DCH}^\ast(-) $ instead of $\CH^\ast(-)_{(0)}$ (\emph{cf.}
		also
		\cite[Rem.~4.15]{FV}, which deals with the same situation).
		Alternatively, one can argue as follows\,:
		according to Porteous' formula \cite[Theorem 15.4]{F}, the difference
		\[ d:= c_i(Z) - b^\ast c_i(X) \ \ \ \in\ \CH^i(Z) \]
		can be expressed in terms of (push-forwards to $Z$ of pullbacks to $E$ of)
		Chern classes of $V$ and Chern classes of the normal bundle of $V$ in $X$. But
		any cycle on $E$ is distinguished (since $E$ has trivial Chow groups), and the
		inclusion morphism $E\to Z$ is distinguished, and hence this difference $d$ is
		distinguished. As the Chern classes $c_i(Z)$ are distinguished, this implies
		that $b^\ast c_i(X)$ is distinguished. Since the morphism $b$ is
		distinguished,
		this implies that
		\[ c_i(X)=b_\ast b^\ast c_i(X)\ \ \ \in\ \operatorname{DCH}^i(X),\]
		\emph{i.e.}, condition $(\star_{\mathrm{Chern}})$ (and hence condition
		$(\star)$) is verified
		for $X$.

		To finish the proof, we observe that the inclusion of (each copy of) $W_1$ in
		the exceptional divisor of $Z_1\to X_1\times X_2$ is distinguished (since both
		$W_1$ and the exceptional divisor have trivial Chow groups). This implies that
		the same is true for the inclusion of (each copy of) $W_1$ in the strict
		transform of this exceptional divisor in $Z_2$. Since the inclusion of the
		exceptional divisor in $Z_2$ is distinguished, this implies that the inclusion
		of (each copy of) $W_1$ in $Z_2$ is distinguished. Since the quotient morphism
		$Z_2\to Z$ and the blow-up $b$ are distinguished, it follows that the
		inclusion
		of $V\cong W_1$ in $X$ is distinguished.
		
		The fact that $(Z,H)$ verifies (ii), plus the fact that $b$ is distinguished,
		guarantees that $(X,H)$ verifies (ii). The fixed locus $\mathrm{Fix}_{X}(H)$
		is
		the disjoint union of the codimension $2$ component $V$, and a divisor (which
		is
		the isomorphic image in $X$ of the exceptional divisor in $Z_2$ lying over the
		codimension $3$ stratum $W_2\subset X_1\times X_2$). In view of the above,
		this
		implies that $(X,H)$ verifies the conditions (ii), (iv) and (v), and so we are
		done.	
	\end{proof}

	\subsection{$\ZZ_{3^c}$-actions}\label{sec:3c}
	
	In this section, we want to show that the inductive approach of Schreieder in
	\cite[\S 8.2]{Schreieder} can be strengthened to take into account the motivic
	structure and to keep track of the condition $(\star)$. For clarity, we follow
	the notations of \cite{Schreieder}.
	
	Precisely, for natural numbers $a\neq b$ and $c\geq 0$, let $S_c^{a,b}$ denote
	the family of pairs $(X,\phi)$, consisting of a smooth projective complex
	variety $X$ of dimension $a+b$ and an automorphism $\phi \in \mathrm{Aut}(X)$
	of
	order $3^c$, such that properties (i)--(v) below hold. Here $\zeta$ denotes a
	fixed primitive $3^c$-th root of unity and $g:= (3^c-1)/2$.
	\begin{enumerate}[(i)]
		\item The decomposition $\h(X) = T \oplus \h(X)^{\langle \phi\rangle}$ is such
		that
		$h^{a,b}(T) =
		h^{b,a}(T) = g$ and $h^{p,q}(T) = 0$ for all other $p\neq q$, and such that
		the
		summand $ \h(X)^{\langle \phi\rangle}$ (which is the $\phi$-invariant part of
		the motive of $X$) is isomorphic to a direct sum of Lefschetz motives $ 
		\bigoplus \mathds{1}(*)$.
		\item The action of $\phi$ on $H^{a,b}(X)$ has eigenvalues
		$\zeta,\ldots,\zeta^g$.
		\item The set $\mathrm{Fix}_X(\phi^{3^{c-1}})$ can be covered by local
		holomorphic charts such that $\phi$ acts on each coordinate function by
		multiplication with some power of $\zeta$.
		\item For $0\leq l \leq c-1$, the motive of $\mathrm{Fix}_X(\phi^{3^l})$ is
		isomorphic to a sum of Lefschetz motives and the action of $\phi$ on that
		motive
		is the identity. 
		\item $X$ has a marking that satisfies the condition $(\star)$ and the
		condition $(\star_{\langle\phi\rangle})$. Moreover, the inclusion morphism
		$\mathrm{Fix}_X(\phi^{3^l}) \hookrightarrow X$ is distinguished for $0\leq l
		\leq c-1$.
	\end{enumerate}
	
	In condition (v), note that it makes sense to say that the inclusion morphism
	is
	distinguished\,: by (iv) the motive of $\mathrm{Fix}_X(\phi^{3^l})$ is
	isomorphic to a sum of Lefschetz motives, and in particular it admits a marking
	that satisfies $(\star)$ (\textit{cf.} \cite[Prop.~5.2]{FV}).
	
	Our condition (i) (resp. (iv)) is a motivic version of conditions (1) and (3)
	(resp. (5)) of Schreieder.  Our conditions (ii) and (iii) are exactly the
	conditions (2) and (4) of Schreieder. The new feature is our condition (v).
	
	As in \cite[\S 8.2]{Schreieder}, we note that it follows from (iii) that
	$\mathrm{Fix}_X(\phi^{3^l})$ is smooth for all $0\leq l \leq c-1$.

	With this strengthened definition of $S_c^{a,b}$ (compared to that of
	\cite{Schreieder}), we have the exact same statement as
	\cite[Proposition~19]{Schreieder}\,:
	\begin{proposition}\label{prop:schreieder}
		Let $(X_1,\phi_1) \in S_c^{a_1,b_1}$ and $(X_2,\phi_2)\in S_c^{a_2,b_2}$. Then
		$$(X_1\times X_2)/\langle \phi_1\times\phi_2\rangle$$
		admits a smooth model $X$ such that the automorphism $\mathrm{id}\times
		\phi_2$
		on $X_1\times X_2$ induces an automorphism $\phi\in \mathrm{Aut}(X)$ with
		$(X,\phi) \in S_c^{a,b}$, where $a=a_1+a_2$ and $b= b_1+b_2$.
	\end{proposition}
	
	Precisely, the variety $X$ is constructed inductively as follows. Consider the
	subgroup of $\mathrm{Aut}(X)$ given by $$G:= \langle \phi_1\times \mathrm{id},
	\mathrm{id}\times \phi_2\rangle.$$ For each $1\leq i \leq c$, consider the
	element of order $3^i$ in $G$ given by $\eta_i :=
	(\phi_1\times\phi_2)^{3^{c-i}}$, generating a cyclic subgroup $G_i:= \langle
	\eta_i\rangle \subseteq G$. We obtain a filtration $$0 = G_0 \subset G_1\subset
	\cdots \subset G_c = \langle \phi_1\times \phi_2\rangle,$$ such that each
	quotient $G_i/G_{i-1}$ is cyclic of order 3, generated by the image of
	$\eta_i$.
	We set $$Y_0:= X_1\times X_2$$ equipped with the natural action of $G$. 
	We define inductively
	\begin{align*}
	Y_{i} &= Y_{i-1}''/\langle\eta_i\rangle,\\
	Y_i' &= \text{Blow up of $Y_i$ along $\mathrm{Fix}_{Y_i}(\eta_{i+1})$},\\
	Y_i''& =\text{Blow up of $Y'_i$ along $\mathrm{Fix}_{Y'_i}(\eta_{i+1})$}.
	\end{align*}
	Here the action of the group $G$ carries at each step. Schreieder shows that
	each $Y_i$ is a smooth model of $Y_{0}/G_i$, so that the variety $X$ of
	Proposition \ref{prop:schreieder} is nothing but $Y_c$ equipped with the action
	of $G/G_c$. To summarize, we have the following diagram\,:
	
	\begin{equation}\label{eq:diagschrei}
	\xymatrix{ & Y''_{c-1} \ar[dl] \ar[dr] & & \cdots  \ar[dl] \ar[dr]  && Y_1'' 
		\ar[dl] \ar[dr] && Y_0''  \ar[dl] \ar[dr] \\
		Y_c && Y_{c-1} && Y_2 && Y_1 && Y_0.
	}
	\end{equation}
	Each arrow to the right corresponds to the composition $Y_i'' \to Y_i' \to Y_i$
	of two blow-ups along fix loci (which turn out to be smooth), and each arrow to
	the left corresponds to a $3-1$ cover, branched along a smooth divisor\,; see
	\cite{Schreieder}. 
	
	\begin{proof}[Proof of Proposition \ref{prop:schreieder}]
		Since there is no point in repeating Schreieder's arguments in full, we only
		indicate how to adapt his proof to show that the motivic statements and the
		condition (v) carry through.
		
		First, since our conditions (i)--(v) imply Schreieder's conditions (1)--(5),
		we
		only need to prove that $X$ satisfies conditions (i), (iv) and (v). Concerning
		the latter two, they are contained in the following strengthening of
		\cite[Lemma~20]{Schreieder}\,:
		\begin{lemma}\label{lem:schreieder}
			Let $\Gamma\subseteq G$ be a subgroup which is not contained in $G_i$. Then
			$\mathrm{Fix}_{Y_i}(\Gamma)$, $\mathrm{Fix}_{Y'_i}(\Gamma)$ and
			$\mathrm{Fix}_{Y''_i}(\Gamma)$ are smooth, their motives are isomorphic to
			direct sums of Lefschetz motives, their $G$-actions restrict to actions on
			each
			irreducible component and the $G_c$-fixed part of their motive is also fixed
			by
			$G$. Moreover, $Y_i$, $Y_i'$ and $Y_i''$ are naturally equipped with markings
			that satisfy $(\star)$ and $(\star_G)$ with the additional property that the
			embeddings $\mathrm{Fix}_{Y_i}(\Gamma) \hookrightarrow Y_i$,
			$\mathrm{Fix}_{Y'_i}(\Gamma)\hookrightarrow Y_i'$ and
			$\mathrm{Fix}_{Y''_i}(\Gamma)\hookrightarrow Y_i''$ are distinguished.
		\end{lemma}
		\begin{proof}
			We follow word-for-word the proof of \cite[Lemma~20]{Schreieder}, which is
			by induction on $i$. The motivic statement is obtained from Schreieder's
			arguments simply by noting the following\,: 
			\begin{enumerate}[(a)]
				\item the fixed locus of $\Gamma$ on $Y_0$ is described as the product of
				fixed
				loci on $X_1$ and $X_2$\,;
				\item the irreducible components of the fixed locus of $\Gamma$
				on $Y_i'$ are described either as  projective bundles over irreducible
				components of a fixed locus on $Y_i$, or as strict transforms of irreducible
				component of some fixed locus, which are themselves blow-ups of irreducible
				components of some fixed locus along the irreducible component of some other
				fixed locus\,;
				\item the irreducible components of the fixed locus of $\Gamma$
				on $Y_i''$ are described similarly as for $Y_i'$\,;
				\item the fixed locus of $\Gamma$ on $Y_{i+1}$ is described either as the
				isomorphic image of a fixed locus on $Y_{i}''$, or its irreducible
				components
				are quotients by $\langle \eta_i \rangle$ of irreducible components of fixed
				loci on $Y_i''$.
			\end{enumerate} 
			In all aforementioned descriptions, the property that the Chow groups are
			trivial is preserved.
			We then note that the motive of a variety is a direct sum of Lefschetz
			motives if and only if its Chow groups are finite-dimensional vector spaces
			if
			and only if the total cycle class map $\CH^*(X) \to \mathrm{H}^*(X)$ is an
			isomorphism\,; see \cite{J2, kimuraFields, vialCRAS}. In particular, assuming
			$Z$ is
			a smooth projective variety with trivial Chow groups, if the
			$G_c$-invariant part of the cohomology of $Z$ is spanned by $G$-invariant
			algebraic cycles, then the $G_c$-invariant part of the motive of this variety
			$Z$ is isomorphic to a direct sum of $G$-invariant Lefschetz motives.
			Together
			with \cite[Lemma~20]{Schreieder}, this establishes the first part of the
			lemma.
			
			Concerning the moreover part, we first recall that any smooth projective
			variety $Z$ whose motive is isomorphic to a direct sum of Lefschetz motives
			satisfies condition $(\star)$ (\textit{cf.} \cite[Prop.~5.2]{FV}). In
			addition, since for any choice of marking we have
			$\operatorname{DCH}^*(Z\times
			Z) = \CH^*(Z\times Z)$, any action of a finite group $G$ on $Z$ satisfies the
			condition $(\star_G)$. 
			
			By induction, assuming that $\mathrm{Fix}_{Y_i}(\Gamma)$ and $Y_i$ have a
			marking satisfying $(\star)$ and $(\star_G)$ such that
			$\mathrm{Fix}_{Y_i}(\Gamma) \hookrightarrow Y_i$ is distinguished, it only
			remains to show that the graphs of the embeddings
			$\mathrm{Fix}_{Y'_i}(\Gamma)\hookrightarrow Y_i'$,
			$\mathrm{Fix}_{Y''_i}(\Gamma)\hookrightarrow Y_i''$
			and $\mathrm{Fix}_{Y_{i+1}}(\Gamma)\hookrightarrow Y_{i+1}$ 
			are distinguished for suitable choices of markings satisfying $(\star)$ and
			$(\star_G)$. In
			fact,
			we only need to show this component-wise for the irreducible components of
			the
			fixed loci of $\Gamma$.
			
			In case (a), which is the initial case, this is obvious (see
			\cite[Prop.~3.5]{FV}).
			
			In case (b) (and also case (c) which is similar), we have the following
			more precise description (\cite[pp. 326--27]{Schreieder}) of the irreducible
			components of the fixed locus of $\Gamma$
			on $Y_i'$. Let $P$ be an irreducible component of
			$\mathrm{Fix}_{Y_i'}(\Gamma)$, and let $Z$ be the image of $P$ inside $Y_i$.
			Then, depending on whether $Z$ is contained in
			$\mathrm{Fix}_{Y_i}(\langle\Gamma,\eta_{i+1}\rangle)$ or not, one of the
			following occurs\,:
			\begin{itemize}
				\item $Z$ is an irreducible component of
				$\mathrm{Fix}_{Y_i}(\langle\Gamma,\eta_{i+1}\rangle)$ and $P\to Z$ is a
				projective sub-bundle of the projective bundle $E_i'|_Z \to Z$, where $E_i'$
				is
				the exceptional divisor of the blow-up $Y_i' \to Y_i$.
				\item $Z$ is an irreducible component of $\mathrm{Fix}_{Y_i}(\Gamma)$ and
				$P$ is the strict transform of $Z$ in $Y_i'$\,; in particular $P \to Z$ is
				the
				blow-up along $\mathrm{Fix}_Z(\eta_i)$.
			\end{itemize}
			In the first case, we have the composition of inclusions $P\hookrightarrow
			E_i' \hookrightarrow Y_i'$. The left inclusion is distinguished because as we
			saw, both $P$ and $E_i'$ have trivial Chow groups. As for the right
			inclusion, $Y_i'$ is the blow-up of $Y_i$ along $\mathrm{Fix}_{Y_i}(\Gamma)$,
			by
			induction $Y_i$ satisfies $(\star)$ and $(\star_G)$, and
			$\mathrm{Fix}_{Y_i}(\Gamma)$ has
			trivial Chow
			groups and has a suitable marking making the inclusion
			$\mathrm{Fix}_{Y_i}(\Gamma) \hookrightarrow Y_i$ distinguished\,; therefore by
			\cite[Prop.~4.8]{FV} $E_i'$ and
			$Y_i'$ have markings that satisfy $(\star)$ and $(\star_G)$ such that $E_i'
			\hookrightarrow Y_i'$ is distinguished.
			In the second case, by arguing as in the proof of
			\cite[Prop.~3.4]{SV} (with $\CH^*(-)_{(0)}$ replaced with
			$\operatorname{DCH}^*(-)$) and using the fact that $P$ has trivial Chow
			groups,  one can show that $P\hookrightarrow Y_i'$ is distinguished.
			
			In case (d), finally, we have that $\pi: Y_i'' \to Y_{i+1}$ is a
			$\ZZ_{3}$-cyclic covering branched along the smooth divisor
			$\mathrm{Fix}_{Y_i''}(\eta_i)$. That $Y_i''$ satisfies $(\star)$ and
			$(\star_G)$, together with the fact proved above (case (c)) that
			$\mathrm{Fix}_{Y_i''}(\eta_i) \hookrightarrow Y_i''$ is distinguished, is enough
			to conclude that the quotient $Y_{i+1}$ satisfies  $(\star)$ and $(\star_G)$\,;
			see \cite[Prop.~4.12]{FV}. 
			It remains to show that $\mathrm{Fix}_{Y_{i+1}}(\eta_i)  \hookrightarrow
			Y_{i+1}$ is distinguished.
			By the projection formula, any generically finite morphism $f: Z_1\to Z_2$ of
			degree $d$ between smooth projective varieties induces a surjective morphism
			$f_* : \h(Z_1)\to \h(Z_2)$ with a section $\frac{1}{d}f^* :\h(Z_2) \to \h(Z_1)$.
			
			If in addition $Z_1$ has a marking, then $f$ is distinguished for the marking
			on $Z_2$ induced by that of $Z_1$. 
			Let $P$ be an irreducible component of $\mathrm{Fix}_{Y_{i+1}}(\Gamma)$. We
			know that there is an irreducible component $Z$ of some fixed locus in $Y_i''$
			such that $\pi$ restricts to either an isomorphism or a 3-to-1 morphism $Z\to
			P$. By induction, $Z$ has a marking such that the inclusion $Z \hookrightarrow
			Y_i''$ is distinguished. Endow $P$ with the marking induced by that of $Z$\,; in
			particular $f$ is distinguished. Then the inclusion $P\hookrightarrow Y_{i+1}$
			is distinguished because it is the composite of  $\pi$, the inclusion
			$Z\hookrightarrow Y_i''$, and $\frac{1}{\deg f}f^*$, all of which are
			distinguished.
			
			The proof of Lemma \ref{lem:schreieder} is complete.
		\end{proof}
		
		We have now established properties (ii)--(v) for $X$. With Lemma
		\ref{lem:schreieder}, we have in fact showed that $$\h(X) \simeq \h(Y_0)^{G_c}
		\oplus \bigoplus \mathds{1}(*),$$ where the right hand side summand is fixed
		by
		the action of $G$. Since the motive of $Y_0$ is of abelian type (and hence
		finite-dimensional in the sense of Kimura), in order to establish (i), it thus
		suffices to see that the $G_c$-invariant cohomology of $Y_0$ is spanned by
		$G$-invariant algebraic classes, by $g$ linearly independent $(a,b)$-forms and
		their conjugates. This follows from conditions (i) and (ii) for $(X_1,\phi_1)$
		and for $(X_2,\phi_2)$, as in \cite[pp. 329--30]{Schreieder}.
		
		The proof of Proposition \ref{prop:schreieder} is now complete.
	\end{proof}

	\begin{remark}\label{diff} As mentioned before (Remark \ref{diff0}), the
		Cynk--Hulek variety $X_{CH}$ (given by Theorem \ref{ch3}) and the Schreieder
		variety $X_S$ (given by the $c=1, b=0$ case of Theorem \ref{schreieder}) are
		not necessarily
		the same. The difference in their construction is clear from comparison of
		subsections \ref{sec:3} and \ref{sec:3c}\,: in the construction of $X_{CH}$,
		there
		is at each step the blow-down $b\colon Z\to X$ (in order to have a {\em
			crepant\/} resolution), whereas in the construction of $X_S$ the varieties
		$Z$
		and $X$ coincide.
	\end{remark}

	\subsection{Proof of the main Theorem \ref{t:bv}}
	In view of Remark \ref{rmk:reduction}, 
	Theorem \ref{t:bv} follows from the following two claims\,:	
	
	\begin{claim}\label{key}
		Let $X$ be the $n$-dimensional Calabi--Yau variety of Theorem \ref{ch} or
		\ref{ch3}, or a Schreieder variety as in Theorem \ref{schreieder}. Then
		$X$ admits a marking $\phi : \h(X) \stackrel{\simeq}{\longrightarrow} M$ that
		satisfies $(\star)$.
	\end{claim}
	
	\begin{claim}\label{key2} Let $X$ be the $n$-dimensional Calabi--Yau variety of
		Theorem \ref{ch} or \ref{ch3}, or a Schreieder variety as in Theorem
		\ref{schreieder}. Then there is a decomposition
		\begin{equation}\label{deco} \h(X)= T\oplus \bigoplus {\mathds{1}}(\ast)\ \ \
		\hbox{in}\ \MM_{\rm rat}\
		,\end{equation}
		where $T$ is such that $H^j(T) = 0$ for $j\neq n$, and $T$ is isomorphic to a
		direct summand of the Chow motive of $E_1\times\cdots\times E_n$ (if $X$ is as
		in Theorem \ref{ch}), resp. of $E^n$ (for $X$ as in Theorem \ref{ch3}), resp. of
		$C^n$ (for $X$ as in Theorem \ref{schreieder}).		
	\end{claim}	
	
	\begin{proof}[Proof of Claim \ref{key}]
		The Cynk--Hulek varieties of Theorem \ref{ch} (resp. Theorem \ref{ch3}) are
		constructed inductively using
		the process of Proposition \ref{p:inductionZ2} (resp. Proposition
		\ref{p:inductionZ3}), by adding an elliptic curve with
		$[-1]$-involution (resp. an elliptic curve with non-trivial $\ZZ_3$-action),
		at
		each step. Repeatedly applying Proposition \ref{p:inductionZ2} (resp.
		Proposition \ref{p:inductionZ3}), 
		we find that they admit a marking satisfying $(\star)$.

		Likewise, the Schreieder varieties are obtained inductively using Proposition
		\ref{prop:schreieder}, by adding the hyperelliptic curve $C_{g,1}$ at each
		step. A repeated application of Proposition
		\ref{prop:schreieder} establishes Claim \ref{key} for the Schreieder
		varieties.  
	\end{proof}
	
	\begin{proof}[Proof of Claim \ref{key2}]
		The Cynk--Hulek varieties of Theorem \ref{ch}  are
		constructed inductively using
		the process of Proposition \ref{p:inductionZ2},  by adding at each step an
		elliptic curve $E$ with
		$[-1]$-involution. Note that the quotient $E/[-1]$ is isomorphic to
		$\mathbb{P}^1$ and hence has Chow motive isomorphic to $\mathds{1} \oplus
		\mathds{1}(-1)$. In particular, the Chow motive of $E$ is isomorphic to $T
		\oplus (\mathds{1} \oplus \mathds{1}(-1))$, where $[-1]$ acts trivially on the
		right-hand side summand. Therefore, in order to prove Claim \ref{key2} for the
		Cynk--Hulek varieties of Theorem \ref{ch}, it is enough to remark that
		Proposition~\ref{p:inductionZ2} continues to hold if one adds the property 
		\begin{enumerate}[(i)]
			\item[\emph{(vi)}] the decomposition $\h(X_i) = T_i\oplus \h(X_i)^{H_i}$ is
			such that $ \h(X_i)^{H_i}	\simeq \bigoplus \mathds{1}(*)$, and $H^j(T_i) = 0$ if
			$j\neq \dim X_i$.
		\end{enumerate}
		Recall that $X$ is the quotient by $G \simeq \ZZ_{2}$ of the blow-up $Z$ of
		$X_1\times X_2$ along $B_1\times B_2$, where $B_i$ is the fixed locus of $H_i$
		acting on $X_i$ and is assumed to have motive isomorphic to a direct sum of
		Lefschetz motives. By the blow-up formula, we have 
		$$\h(Z) \simeq \h(X_1)\otimes \h(X_2) \oplus \big(\h(B_1)\otimes\h(B_2) \oplus
		\h(B_1)\otimes\h(B_2)(-1)\big).$$
		The right-hand side summand is fixed under the action of $H_1\times H_2$ and is
		isomorphic to a direct sum of Lefschetz motives. Thus in order to conclude it is
		enough to note that 
		$(T_1\otimes \h(X_2)^{H_2})^G=0$ (and similarly $\h(X_1)^{H_1}\otimes
		T_2)^{G}=0$) and
		the $(H_1\times H_2)$-invariant part of the motive $\h(X_1)\otimes \h(X_2)$ is
		a direct sum of Lefschetz motives\,; this follows at once from the assumption
		that $T_i^{H_i} = 0$.  
		
		The proof of Claim \ref{key2} for the Schreieder varieties was already taken
		care of. Indeed, thanks to Proposition \ref{prop:schreieder} we know
		that Schreieder varieties $X$ are in the class $S_c^{a,b}$; this entails in
		particular (by definition of $S_c^{a,b}$) that
		the motive of $X$ decomposes as
		\[ \h(X) =T \oplus  \bigoplus \mathds{1}(*)    \ \ \ \hbox{in}\ \MM_{\rm
			rat}\ ,\]
		where $T$ is such that $H^j(T)=0$ for all $j\not=\dim X$. In fact, at the end
		of the proof of Proposition \ref{prop:schreieder}), we
		constructed $T$ as a direct summand of $ \h(Y_0)^{G_c} $,
		and so $T$ is indeed a direct summand of the Chow motive of $C^n$.		
		
		Finally, to prove Claim \ref{key2} for the Cynk--Hulek varieties $X_{CH}$ of
		Theorem \ref{ch3} we argue as follows\,: the variety $X_{CH}$ is dominated by
		the
		Schreieder variety $X_S$ (of the same dimension, and where $c=1$ and $b=0$ in
		the Schreieder construction). Thus, the truth of Claim \ref{key2} for $X_S$
		implies the truth of Claim \ref{key2} for $X_{CH}$.
	\end{proof}

	\subsection{Proof of Theorem \ref{t:mck}} This is immediate\,:
	in view of Proposition \ref{prop:multmarking}, Theorem \ref{t:mck} follows from
	Claim~\ref{key}.\qed

	\begin{remark}
		Alternatively, we could have established the existence of a self-dual
		multiplicative Chow--K\"unneth decomposition for $X$ as in Theorem \ref{t:mck}
		directly (\emph{i.e.}, without invoking Proposition~\ref{prop:multmarking}
		(\cite[Prop.~6.1]{FV}) by using the
		results of \cite{SV} instead of those of \cite{FV}.
	\end{remark}

	\subsection{Proof of Corollary \ref{cor:moddiag}} By Claim \ref{key}, we know
	that any Schreieder surface $S$ has a marking that satisfies $(\star)$. Since
	the cycle $$(x,x,x) - (x,x,p) - (x,p,x) - (p,x,x) +(p,p,x) + (p,x,p) + (x,x,p)$$
	is numerically trivial (this cycle is numerically trivial for any regular
	surface), it suffices to show that there exists a point $p$ in $S$ such that
	each summand is distinguished. By definition of $(\star_{\mathrm{Mult}})$, the
	cycle $(x,x,x)$ is distinguished. By \cite[Lemma~3.8]{FV}, the diagonal
	$\Delta_S = (x,x)$ is distinguished in $S\times S$\,; and it is clear that the
	fundamental class of $S$ is distinguished in $S$ (see also
	\cite[Remark~3.4]{FV}). Therefore it suffices to exhibit a point $p$ that is
	distinguished in $S$. The variety $S$ is constructed using the diagram
	\eqref{eq:diagschrei} starting from $Y_0$ the product of two hyperelliptic
	curves. Choose a fixed point $p_0$ of $G$ in $Y_0$. By Proposition
	\ref{prop:markinghyperelliptic} and the fact that $(\star)$ is stable under
	product, $Y_0$ has a marking that satisfies $(\star)$ such that $p_0$ is
	distinguished. Now by inspection of the proof of Proposition
	\ref{prop:schreieder}, we see, by defining inductively $p_i$ as the image of
	$p_{i-1}''$ under $Y_{i-1}'' \to Y_i$, $p_i'$ as a point in the pre-image of
	$p_i$ under $Y_i' \to Y_i$, and $p_i''$  as a point in the pre-image of $p_i'$
	under $Y_i'' \to Y_i'$, that $p_c = p$ is a distinguished point in $S = Y_c$.
	
	Alternately, Corollary \ref{cor:moddiag} is a consequence of Theorem
	\ref{t:mck} combined with \cite[Proposition~8.14]{SVfourier} and with the fact
	that there exists a point $p$ in $S$ that is distinguished.\qed

	\subsection{Final remarks}
	We remark that Theorems \ref{t:bv} and \ref{t:mck} hold also in the following
	two situations. First in the Schreieder construction, we may add via Proposition
	\ref{prop:schreieder} at each step, instead of the hyperelliptic curve
	$C_{g,1}$, more generally the hyperelliptic curve $C_{g,D}$ for any non-zero
	rational number $D$. Second in the construction of a smooth model of the
	Cynk--Hulek varieties of Theorem \ref{ch}, one may add via Proposition
	\ref{p:inductionZ2} at each step a hyperelliptic curve equipped with its
	hyperelliptic involution instead of an elliptic curve equipped with its
	$[-1]$-involution.

	\section{Applications}\label{sec:applications}
	
	\subsection{Voisin's conjecture}
	\label{svois}
	
	Voisin \cite{V9} has formulated the following intriguing conjecture, which is a
	special instance of the
	Bloch--Beilinson conjectures.
	
	\begin{conjecture}[Voisin \cite{V9}]\label{conj} Let $X$ be a smooth projective
		variety of dimension $n$, with $p_g(X):= h^{n,0}(X) = 1$ and $h^{j,0}(X)=0$
		for
		$0<j<n$. Then any two zero-cycles $a,a^\prime\in \CH^n_{num}(X)$ satisfy
		\[ a\times a^\prime = (-1)^n a^\prime\times a\ \ \ \hbox{in}\ \CH^{2n}(X\times
		X)\ .\]
		(Here, $a\times a^\prime$ is the exterior product $(p_1)^\ast(a)\cdot
		(p_2)^\ast(a^\prime)\in \CH^{2n}(X\times X)$, where $p_j$ is projection to the
		$j$-th factor.)
	\end{conjecture}

	For background and motivation for Conjecture \ref{conj}, \emph{cf.}
	\cite[Section 4.3.5.2]{Vo}.
	Conjecture \ref{conj} has been proven in some scattered special cases
	\cite{V9},
	\cite{moi}, \cite{desult}, \cite{tod}, \cite{BLP}, but is still widely open for
	a general K3 surface.
	
	\begin{remark}
		Conjecture \ref{conj} can be thought of as a version of Bloch's conjecture for
		motives.
		Indeed, given $X$ as in Conjecture \ref{conj}, consider the Chow motive $M$
		defined as
		\[ M:= \begin{cases}  \wedge^2 \h^n(X) :=  (X\times X, {1\over
			2}{\displaystyle \sum_{\sigma\in\Sy_2}}\hbox{sgn}(\sigma) \Gamma_\sigma\circ
		(\pi^n_X\times\pi^n_X),0)      &\hbox{if\ $n$\ is\ even}\ ,\\
		\hbox{Sym}^2 \h^n(X) :=    (X\times X, {1\over
			2}{\displaystyle\sum_{\sigma\in\Sy_2}}  \Gamma_\sigma\circ
		(\pi^n_X\times\pi^n_X),0)           &\hbox{if\ $n$\ is\ odd}\ .\\
		\end{cases}\]
		(Here, for $\h^n(X)$ to make sense, we need to assume that $X$ has a
		Chow--K\"unneth decomposition, in the sense of Definition \ref{ck}).
		The condition on $p_g(X)$ implies that $h^{2n,0}(M)= 0$, and so $M$ is a
		motive
		with
		\[ h^{j,0}(M)=0\ \ \ \forall\ j\ .\]
		A motivic version of Bloch's conjecture would then imply that
		\[ \CH_0^{}(M)=0\ .\]
		On the other hand, the condition on $h^{j,0}(X)$ conjecturally implies that
		$\CH_{num}^n(X)=(\pi^n_X)_\ast \CH^n(X)$. It follows that given two
		zero-cycles 
		$a,a^\prime\in \CH^n_{num}(X)$, one conjecturally has
		\[   a\times a^\prime - (-1)^n a^\prime\times a = (\pi^n_X\times \pi^n_X)_\ast
		(a\times a^\prime) -   (-1)^n \iota_\ast  (\pi^n_X\times \pi^n_X)_\ast
		(a\times
		a^\prime)  \ \ \mbox{in}\ \CH_0^{}(M)=0,\]
		where $\iota$ is the non-trivial element of $\Sy_2$.
		This heuristically explains Conjecture \ref{conj}.
	\end{remark}
	
	We now prove Voisin's conjecture for Cynk--Hulek Calabi--Yau varieties\,: 
	
	\begin{theorem}\label{conjV} Let $X$ be a Calabi--Yau variety of dimension $n$
		as in Theorem \ref{ch} or \ref{ch3}. Then Conjecture \ref{conj} is true for
		$X$:
		any $a,a^\prime\in \CH^n_{num}(X)$ satisfy
		\[ a\times a^\prime = (-1)^n \, a^\prime\times a\ \ \ \hbox{in}\
		\CH^{2n}(X\times X)\ .\]
	\end{theorem}
	\begin{proof} Consider morphisms
		$$\xymatrix{
			& E_1\times\cdots\times E_n \ar[d]^p \\
			X \ar[r]^f & \bar{X}
		}$$
		as in Theorem \ref{ch} or \ref{ch3}. 
		The Chow group of $0$-cycles is a birational invariant amongst varieties that
		are global quotients (this follows for instance from \cite[Example
		17.4.10]{F}),
		and so
		$ f^\ast\colon \CH^n(\bar{X}) \to \CH^n(X)\ $  is an isomorphism.
		Consequently, it suffices to prove Conjecture~\ref{conj} for $\bar{X}$. 
		Using Corollary \ref{C:CH},
		we see
		that $\CH^n_{num}(\bar X)$ is contained in $p_*\CH^n(E_1\times \cdots
		E_n)_{(n)}$.
		Therefore, we are reduced  to proving that $a\times a^\prime = (-1)^n
		a^\prime\times a$ for all $a, a' \in \CH^n(E_1\times \cdots \times
		E_n)_{(n)}$\,; this is a special case of\,:
		\begin{proposition}[Voisin, Example 4.40 in \cite{Vo}]\label{abvar} 
			Let $B$ be an abelian variety of dimension $n$. Let $a,a^\prime\in
			\CH^n(B)_{(n)}$. Then
			\[a\times a^\prime=(-1)^n \, a^\prime\times a\ \ \ \hbox{in}\
			\CH^{2n}(B\times B) .\]
		\end{proposition}
		This concludes the proof of the theorem.
	\end{proof}

	\subsection{Voevodsky's conjecture}
	\label{svoe}
	
	In this paragraph, we give an application of our results to Voevodsky's
	conjecture on smash-equivalence.

	\begin{definition}[Voevodsky \cite{Voe}]\label{sm} Let $X$ be a smooth
		projective variety. A cycle $a\in \CH^i(X)$ is called {\em smash-nilpotent\/} 
		if there exists $m\in\NN$ such that
		\[ \begin{array}[c]{ccc}  a^m:= &\underbrace{a\times\cdots\times a}_{(m\hbox{
				times})}&=0\ \ \hbox{in}\ 
		\CH^{mi}(X\times\cdots\times
		X)_{}\ .
		\end{array}\]
		\vskip0.6cm
		
		Two cycles $a,a^\prime$ are called {\em smash-equivalent\/} if their
		difference
		$a-a^\prime$ is smash-nilpotent. We will write $\CH^i_\otimes(X)\subseteq
		\CH^i(X)$ for the subgroup of smash-nilpotent cycles.
	\end{definition}
	
	\begin{conjecture}[Voevodsky \cite{Voe}]\label{voe} Let $X$ be a smooth
		projective variety. Then
		\[  \CH^i_{num}(X)\ \subseteq\ \CH^i_\otimes(X)\ \ \ \hbox{for\ all\ }i\ .\]
	\end{conjecture}
	
	\begin{remark} It is known \cite[Th\'eor\`eme 3.33]{An} that Conjecture
		\ref{voe} for all smooth projective varieties implies (and is strictly
		stronger
		than) Kimura's conjecture ``all smooth projective varieties have
		finite-dimensional motive'' \cite{Kim}.
	\end{remark}

	Thanks to Claim \ref{key2}, we can verify Voevodsky's
	conjecture for odd-dimensional Cynk--Hulek varieties and Schreieder
	varieties\,:

	\begin{proposition}\label{conjVoe} Let $X$ be a Cynk--Hulek Calabi--Yau variety
		as in Theorem \ref{ch} or \ref{ch3}, or a Schreieder variety as in Theorem
		\ref{schreieder}. Suppose the dimension $n$ of $X$ is odd.
		Then
		\[  \CH^i_{num}(X)\ \subseteq\ \CH^i_\otimes(X)\ \ \ \hbox{for\ all\ }i.\]
	\end{proposition}
	
	\begin{proof} 
		According to Claim \ref{key2}, we have a decomposition
		\[ \h(X) = T\oplus \bigoplus {\mathds{1}}(\ast), \]	
		with $H^j(T)=0$ for $j\not= n$, and $T$ isomorphic to a direct summand of
		$\h(C_1\times\cdots\times C_n)$. Here, the $C_i$ are elliptic curves in case $X$
		is a Cynk--Hulek variety, and the hyperelliptic curves of \S\ref{sec:Cg} in case
		$X$ is a Schreieder variety.

		By Kimura finite-dimensionality, $T$ is isomorphic to a direct summand of the
		motive
		$(C_1\times\cdots\times C_n,\pi^n,0)$, where $\pi^n$ is any Chow--K\"unneth
		projector on the degree-$n$ cohomology.		
		But the Chow motive $(C_1\times\cdots\times C_n,\pi^n,0)$ is oddly
		finite-dimensional (in the sense of \cite{Kim}).
		Hence, together with the fact that $\CH^{i}_{num}({X})=
		\CH^i_{num}(T)$, the corollary is implied by the fact that $\CH^i_{}(M)
		\subseteq \CH^i_\otimes(M)$ for all $i$ and for all oddly finite-dimensional
		Chow motives $M$  (this is due to Kimura
		\cite[Proposition 6.1]{Kim}, and is also used in \cite{KSeb}). 
	\end{proof}

	\subsection{Supersingularity} 
	
	The construction of the Cynk--Hulek Calabi--Yau varieties also makes sense in
	positive characteristic $\ge 5$. In this final section, we present
	supersingular
	Calabi--Yau varieties for which the motive behaves in stark contrast to the
	characteristic zero case\,:
	
	\begin{proposition}\label{ss} Let $k$ be an algebraically closed field of
		characteristic $\ge 5$. Let $X$ be a Calabi--Yau variety over $k$ obtained as
		in
		Theorem \ref{ch} or \ref{ch3}, where the elliptic curves are assumed to be
		supersingular. Assume $X$ is even-dimensional. Then the Chow motive of $X$ is
		isomorphic to a direct sum of Lefschetz motives. Consequently, the cycle class
		map to $\ell$-adic cohomology induces
		isomorphisms
		\[  \CH^i(X)_{\QQ_\ell}\ \xrightarrow{\cong}\ H^{2i}(X, \QQ_\ell(i))\ \ \
		\forall i\ \]
		(where $\ell$ is a prime different from $\hbox{char}(k)$). 
	\end{proposition}
	
	\begin{proof} 
		First of all, we observe that the construction of the smooth projective
		Calabi--Yau varieties of Cynk--Hulek carries over to characteristic $\ge 5$.
		Using Claim \ref{key2}, we have a decomposition
		\[ \h(X) = T\oplus \bigoplus {\mathds{1}}(\ast), \]	
		with $H^j(T)=0$ for $j\not= n$, and $T$ isomorphic to a direct summand of
		$\h(E_1\times\cdots\times E_n)$.
		By Kimura finite-dimensionality, $T$ is isomorphic to a direct summand of the
		motive
		$(E_1\times\cdots\times E_n,\pi^n,0)$, where $\pi^n$ is any Chow--K\"unneth
		projector on the degree-$n$ cohomology.	
		Therefore, by Kimura finite-dimensionality again, one is reduced to proving that
		$(E_1\times\cdots\times E_n,\pi^n,0)$ is isomorphic to a direct sum of Lefschetz
		motives when the elliptic curves $E_1,\ldots, E_n$ are supersingular.
		
		Recall that there is only one isogeny class of supersingular elliptic curve\,;
		let us denote it by $E$. 
		We endow $E$ with its canonical Chow--K\"unneth decomposition, namely the one
		given by $\h^i(E) = (E,\pi_E^i,0)$ with
		\[\pi_E^0 := 0_E\times E, \quad \pi_E^2 = E\times 0_E, \quad \pi^1_E = \Delta_E
		- \pi_E^0 - \pi_E^2.\]		
		Since $E$ is supersingular, we have that
		$\mathrm{End}(\h^1(E))=\mathrm{End}(E)\otimes \QQ$ is $4$-dimensional. This
		implies that 
		\begin{equation}\label{eq:supersingular}
		\h^1(E) \otimes \h^1(E) \simeq \mathds{1}(-1)^{\oplus 4}.
		\end{equation}
		Now $\pi_{E^n}^n$ can be chosen as follows\,:
		$$\pi_{E^n}^n = \sum_{i_1+\cdots+i_n = n} \pi_E^{i_1}\otimes \cdots \otimes
		\pi_E^{i_n}.$$ But for $i_1+\cdots+i_n $ even, 
		$(E^n,\pi_E^{i_1}\otimes \cdots \otimes \pi_E^{i_n},0) =
		\h^{i_1}(E)\otimes\cdots \otimes \h^{i_n}(E)$ is isomorphic to a direct sum of
		Lefschetz motives thanks to \eqref{eq:supersingular}. Consequently, for $n$
		even, $(E^n, \pi_{E^n}^n,0)$ is isomorphic to a direct sum of Lefschetz motives,
		thereby establishing the proposition. 
	\end{proof}

	\begin{remark} In dimension $n=2$, Proposition \ref{ss} follows from the
		general
		result that supersingular K3 surfaces are unirational \cite{Lied}. For the
		proof of Proposition~\ref{ss}, we could alternatively have used the description
		of the Chow groups of supersingular abelian varieties given by Fakhruddin
		\cite{Fak}. Finally, in case the ground field $k$ is finite, the fact that the
		cycle class map to $\ell$-adic cohomology induces
		isomorphisms
		$\CH^i(X)_{\QQ_\ell}\ \xrightarrow{\cong}\ H^{2i}(X, \QQ_\ell(i))$ for all $i$
		and for all Cynk--Hulek Calabi--Yau varieties is true without the
		assumption
		of supersingularity and also without the restriction on the parity of the
		dimension, as follows from \cite{Kahn} (but beware that the Chow motive is then
		not necessarily isomorphic to a direct sum of Lefschetz motives, even in the
		even-dimensional case).
	\end{remark}

\end{document}